\documentclass[12pt,reqno]{amsart}
\usepackage{pgf}
\usepackage{tikz}

\usetikzlibrary{arrows, decorations.pathmorphing, backgrounds, positioning, fit, petri, automata}
\usepackage{amsmath,amssymb,graphicx,mathrsfs,bm}

\usepackage{times}

\tikzset{font={\fontsize{10pt}{12}\selectfont}}
\usepackage[colorlinks=true,allcolors = blue]{hyperref}

\numberwithin{equation}{section}
\newtheorem{thm}{Theorem}[section]
\newtheorem{prop}[thm]{Proposition}
\newtheorem{lem}[thm]{Lemma}
\newtheorem{cor}[thm]{Corollary}

\theoremstyle{definition}
\newtheorem{rem}[thm]{Remark}
\newtheorem{eg}[thm]{Example}

\renewcommand{\theequation}{\thesection.\arabic{equation}}

\newcommand{\beq}{\begin{equation}}
\newcommand{\eeq}{\end{equation}}

\newcommand{\C}{{\mathbb C}}
\newcommand{\Z}{{\mathbb Z}}

\newcommand{\R}{{\mathbb R}}

\newcommand{\bean}{\begin{eqnarray}}
\newcommand{\eean}{\end{eqnarray}}
\newcommand{\be}{\begin{displaymath}}
\newcommand{\ee}{\end{displaymath}}

\newcommand{\cB}{\mathcal{B}}
\newcommand{\cDB}{\mathcal{D^B}}

\newcommand{\cM}{\mathcal{M}}

\newcommand{\bP}{\mathbb{P}}

\newcommand{\rdet}{\mathrm{rdet}}

\newcommand{\bs}{\boldsymbol}

\newcommand{\g}{\mathfrak{g}}
\newcommand{\h}{\mathfrak{h}}
\newcommand{\n}{\mathfrak{n}}

\newcommand{\gl}{\mathfrak{gl}}
\newcommand{\sing}{\mathrm{sing}}
\newcommand{\sln}{\mathfrak{sl}}

\newcommand{\bLa}{{\boldsymbol \La}}
\newcommand{\Gr}{{\mathrm{Gr}}}
\newcommand{\Wr}{{\mathrm{Wr}}}

\newcommand{\la}{\lambda}
\newcommand{\La}{\Lambda}
\newcommand{\s}{\mathrm{s}}
\newcommand{\sGr}{\mathrm{sGr}}

\newcommand{\Ker}{\mathop{\rm Ker}}

\newcommand{\Tr}{{\rm Tr}}

\newcommand{\gge}{\geqslant}
\newcommand{\lle}{\leqslant}
\newcommand{\opg}{\mathrm{op}_{\g}}
\newcommand{\Opg}{\mathrm{Op}_{\g}}

\newcommand{\calpha}{\check{\alpha}}
\newcommand{\crho}{\check{\rho}}
\newcommand{\cmu}{\check{\mu}}
\newcommand{\cla}{\check{\lambda}}

\newcommand{\on}{\operatorname}
\newcommand{\mc}{\mathcal}
\newcommand{\tl}{\tilde}
\newcommand{\mb}{\mathbb}
\newcommand{\pa}{{\partial}}
\makeatletter
\def\@eqnnum{{\normalfont \color{red} (\theequation)}}
\makeatother

\begin{document}
	\pagestyle{myheadings}
	
	\setcounter{page}{1}
	
	\title[Self-dual Grassmannian]{\hspace{45pt} Self-dual Grassmannian, Wronski map, \newline and representations of $\gl_N$, ${\mathfrak{sp}}_{2r}$, ${\mathfrak{so}}_{2r+1}$}
	
	\author{Kang Lu, E. Mukhin, and A. Varchenko}
	\address{K.L.: Department of Mathematical Sciences,
		Indiana University-Purdue University Indianapolis,
		402 N.Blackford St., LD 270,
		Indianapolis, IN 46202, USA}\email{lukang@iupui.edu}
	\address{E.M.: Department of Mathematical Sciences,
		Indiana University-Purdue University Indianapolis,
		402 N.Blackford St., LD 270,
		Indianapolis, IN 46202, USA}\email{emukhin@iupui.edu}
	\address{A.V.: Department of Mathematics, University of North Carolina at Chapel Hill,
		Chapel Hill, NC 27599-3250, USA}
	\email{anv@email.unc.edu}

\dedicatory{Dedicated to Yuri Ivanovich Manin on the occasion of his 80-th birthday}
	
\begin{abstract}
We define a $\gl_N$-stratification of the Grassmannian of $N$ planes $\Gr(N,d)$.  The $\gl_N$-stratification consists of strata $\Omega_{\bLa}$ labeled by unordered sets $\bLa=(\la^{(1)},\dots,\la^{(n)})$ of nonzero partitions with at most $N$ parts, satisfying a condition depending on $d$, and such that
$(\otimes_{i=1}^n V_{\la^{(i)}})^{\sln_N}\ne 0$.  
Here $V_{\la^{(i)}}$ is the irreducible $\gl_N$-module with highest weight $\la^{(i)}$. 
We show that the closure of a stratum $\Omega_{\bLa}$ is the union of the strata
$\Omega_{\bm\Xi}$,  $\bm\Xi=(\xi^{(1)},\dots,\xi^{(m)})$, such
 that there is a partition $\{I_1,\dots,I_m\}$ of  $\{1,2,\dots,n\}$ with
$	{\rm {Hom}}_{\gl_N} (V_{\xi^{(i)}}, \otimes_{j\in I_i}V_{\la^{(j)}}\big)\neq 0$ for $i=1,\dots,m$. The $\gl_N$-stratification of the Grassmannian agrees with the Wronski map. 

We introduce and study the new object: the self-dual Grassmannian
$\sGr(N,d)\subset \Gr(N,d)$. Our main result is a similar $\g_N$-stratification of the self-dual Grassmannian governed
by representation theory of the Lie algebra $\mathfrak g_{2r+1}:=\mathfrak{sp}_{2r}$ if $N=2r+1$ and of the Lie algebra $\mathfrak g_{2r}:=\mathfrak{so}_{2r+1}$ if $N=2r$. 

\end{abstract}

	\maketitle
	
	\section{Introduction}
The Grassmannian $\Gr(N,d)$ of $N$-dimensional subspaces of the complex $d$-dimensional vector space has the standard
stratification by Schubert cells $\Omega_\la$  labeled by partitions 
$\la=
(d-N\gge\la_1\gge\dots\gge \la_N\gge 0)$. A Schubert cycle is the closure of a cell $\Omega_\la$. It is well known that the Schubert cycle $\overline{\Omega}_\la$ is the union of the cells $\Omega_\xi$  such that
the Young diagram of $\la$ is inscribed into the Young diagram of $\xi$. This stratification depends on a choice of a full flag in the $d$-dimensional space. 

In this paper we introduce a new
stratification of $\Gr(N,d)$ governed by
representation theory of $\gl_N$ and called the {\it $\gl_N$-stratification}, see Theorem \ref{thm gr dec}.  The $\gl_N$-strata 
$\Omega_\bLa$ are
labeled by unordered sets $\bLa=(\la^{(1)},\dots,\la^{(n)})$ of nonzero partitions 
$\la^{(i)}=(d-N\gge\la^{(i)}_1\gge\dots\gge \la^{(i)}_N\gge 0)$ such that
\beq\label{1}
\quad
(\otimes_{i=1}^n V_{\la^{(i)}})^{\sln_N}\ne 0, \qquad \sum_{i=1}^n\sum_{j=1}^N\la^{(i)}_j=N(d-N),
\eeq
where $V_{\la^{(i)}}$ is the irreducible $\gl_N$-module with highest weight
$\la^{(i)}$. We have $\dim \Omega_\bLa = n$. We call the closure of a stratum $\Omega_\bLa$ in $\Gr(N,d)$ a {\it $\gl_N$-cycle}. The $\gl_N$-cycle $\overline{\Omega}_\bLa$ is an algebraic set in $\Gr(N,d)$. We show that $\overline{\Omega}_\bLa$ is the union of the strata
$\Omega_{ \bm\Xi}$,  $\bm\Xi=(\xi^{(1)},\dots,\xi^{(m)})$, such that there is a partition $\{I_1,\dots,I_m\}$ of  $\{1,2,\dots,n\}$ with
$	{ \rm {Hom}}_{\gl_N} (V_{\xi^{(i)}}, \otimes_{j\in I_i}V_{\la^{(j) }}\big)\neq 0$ for $i=1,\dots,m$, see Theorem \ref{thm A strata}.

Thus we have a partial order on the set of sequences of partitions satisfying \eqref{1}. Namely $\bLa\gge \bm \Xi$ if there is a partition $\{I_1,\dots,I_m\}$ of  $\{1,2,\dots,n\}$ with
$	{ \rm {Hom}}_{\gl_N} (V_{\xi^{(i)}}, \otimes_{j\in I_i}V_{\la^{(j) }}\big)\neq 0$ for $i=1,\dots,m$. An example of the corresponding graph is given in Example \ref{A ex}. The $\gl_N$-stratification can be viewed as the geometrization of this partial order.

Let us describe the construction of the strata in more detail. We identify the Grassmannian $\Gr(N,d)$ with the Grassmannian of 
$N$-dimensional subspaces of the $d$-dimensional space $\C_d[x]$
of polynomials in $x$ of degree less than $d$. In other words, we always assume that for $X\in\Gr(N,d)$, we have $X\subset \C_d[x]$. Set $\bP^1=\C\cup\{\infty\}$. Then, for any $z\in {\Bbb {P}}^1$, we have the osculating flag $\mc F(z)$, see  \eqref{F(inf)}, \eqref{F(z)}. Denote the 
Schubert cells corresponding to  $\mc F(z)$ by $\Omega_{\la}(\mc F(z))$. Then the stratum $\Omega_{\bs \La}$ consists of spaces $X\in\Gr(N,d)$ such that $X$ belongs to the intersection of  Schubert cells $\Omega_{\la^{(i)}}(\mc F(z_i))$ for some choice of distinct $z_i\in{\Bbb {P}}^1$:
$$
\Omega_{\bs \La}=\underset{\underset{z_i\neq z_j}{z_1,\dots,z_n}}{\bigcup} \Big(\bigcap_{i=1}^n \Omega_{\la^{(i)}}(\mc F(z_i))\Big)\subset \Gr(N,d).
$$
A stratum $\Omega_{\bLa}$ is a ramified covering over $({\Bbb {P}}^1)^n$ without diagonals quotient by the free action of an appropriate symmetric group, see Proposition \ref{prop deg of cover A}. The degree of the covering is $\dim (\otimes_{i=1}^n V_{\la^{(i)}})^{\sln_N}$. 

For example, if $N=1$, then $\Gr(1,d)$ is the $(d-1)$-dimensional projective space of the vector space
$\C_d[x]$. The strata $\Omega_{\bs m}$ are labeled by unordered sets $\bs m=(m_1,\dots,m_n)$ of positive integers
such that $m_1+\dots+m_n=d-1$.
A stratum $\Omega_{\bs m}$ consists of all polynomials $f(x)$ which have $n$ distinct zeros of multiplicities 
$m_1,\dots,m_n$. In this stratum we also include the polynomials of degree 
 $d-1-m_i$ with $n-1$ distinct roots of multiplicities
 $m_1,\dots,m_{i-1},m_{i+1},\dots ,m_n$. We interpret these polynomials as having a zero of multiplicity $m_i$ at infinity. 
The stratum $\Omega_{(1,\dots,1)}$ is open in $\Gr(1,d)$. The union of other strata is classically called the {\it swallowtail} and the $\gl_1$-stratification is the standard stratification of the swallowtail, see for example Section 2.5 of Part 1 of \cite{AGV}. 

\medskip 

The $\gl_N$-stratification of $\Gr(N,d)$ agrees with the Wronski map
$$\Wr : \Gr(N,d) \to \Gr(1, N(d-N)+1)$$ which sends an $N$-dimensional subspace of polynomials to its Wronskian
$\det (d^{i-1}f_j/dx^{i-1})_{i,j=1}^N$, where $f_1(x),\dots, f_N(x)$ is a basis of the subspace. 
For any $\gl_1$-stratum $\Omega_{\bs m}$ of 
$\Gr(1, N(d-N)+1)$, the preimage  of $\Omega_{\bs m}$ under the Wronski map is the union of $\gl_N$-strata of  $\Gr(N,d)$ and the restriction of the Wronski map to each of those
strata $\Omega_\bLa$  is a ramified covering over $\Omega_{\bs m}$ of degree $b(\bLa)\dim (\otimes_{i=1}^n V_{\la^{(i)}})^{\sln_N},$ where
$b(\bLa)$ is some combinatorial symmetry coefficient of $\bLa$, see \eqref{sym co}.

\medskip

The main goal of this paper is to develop a similar picture for the new object 
$\sGr(N,d)\subset \Gr(N,d)$, called {\it self-dual Grassmannian}. Let $X\in\Gr(N,d)$ be an $N$-dimensional subspace of polynomials in $x$.
Let $X^\vee$ be the $N$-dimensional space of polynomials which are Wronski determinants of $N-1$ elements of $X$:
$$
X^\vee=\{\det\left(d^{i-1}f_j/dx^{i-1}\right)_{i,j=1}^{N-1}, \  f_j(x)\in X \}.
$$
The space $X$ is called \emph{self-dual} if $X^\vee= g\cdot X$ for some polynomial $g(x)$, see \cite{MV}. We define $\sGr(N,d)$ as the subset of
$\Gr(N,d)$ of all self-dual spaces. It is an algebraic set.

The main result of this paper is the stratification of
 $\s\Gr(N,d)$ governed by representation theory of the Lie algebras
$\mathfrak g_{2r+1}:=\mathfrak{sp}_{2r}$ if $N=2r+1$ and $\mathfrak g_{2r}:=\mathfrak{so}_{2r+1}$ if $N=2r$.
This stratification of $\s\Gr(N,d)$ is called the {\it $\g_N$-stratification}, see Theorem \ref{thm sgr dec}.

The $\g_N$-stratification of $\s\Gr(N,d)$ consists of $\g_N$-strata $\s\Omega_{\bLa,\bs k}$ labeled by unordered sets
 of dominant integral $\g_N$-weights $\bLa=(\la^{(1)},\dots,\la^{(n)})$, equipped with
nonnegative integer labels $\bs k=(k_1,\dots,k_n)$, such that
$(\otimes_{i=1}^n V_{\la^{(i)}})^{\mathfrak g_N}\ne 0$ and satisfying a condition similar to the second equation in (\ref{1}), see Section  
\ref{sec prop strata}. Here $V_{\la^{(i)}}$ is the irreducible
$\g_N$-module with highest weight $\la^{(i)}$. Different liftings of an $\sln_N$-weight to a $\gl_N$-weight differ by a vector $(k,\dots,k)$ with integer $k$. Our label $k_i$ is an analog of this parameter in the case of $\g_N$. 

A $\g_N$-stratum $\s\Omega_{\bLa,\bm k}$ is a ramified covering over $({\Bbb {P}}^1)^n$ without diagonals quotient by the free action of an appropriate symmetric group. The degree of the covering is $\dim (\otimes_{i=1}^n V_{\la^{(i)}})^{\g_N}$ and, in particular,
$\dim \s\Omega_{\bLa,\bs k}=n$, see Proposition \ref{prop bij BC}. We call the closure of a stratum $\s\Omega_{\bLa,\bs k}$ in $\s\Gr(N,d)$ a {\it $\g_N$-cycle}. The $\g_N$-cycle $\overline{\s\Omega}_{\bLa,\bs k}$ is an algebraic set. We show that $\overline{\s\Omega}_{\bLa,\bs k}$ is the union of the strata
$\s\Omega_{ \bm\Xi,\bs l}$,  $\bm\Xi=(\xi^{(1)},\dots,\xi^{(m)})$, such that there is a partition $\{I_1,\dots,I_m\}$ of  $\{1,2,\dots,n\}$ satisfying
$	{ \rm {Hom}}_{\g_N} (V_{\xi^{(i)}}, \otimes_{j\in I_i}V_{\la^{(j) }}\big)\neq 0$ for $i=1,\dots,m$, and the appropriate matching of labels, see Theorem \ref{thm BC strata}.

If  $N=2r$, there is exactly one stratum of top dimension  $2(d-N)=\dim \s\Gr(N,d)$. 
For example, the $\mathfrak{so}_5$-stratification of $\sGr(4,6)$ consists of 9 strata of dimensions 4, 3, 3, 3, 2, 2, 2, 2, 1, see the graph of adjacencies in Example \ref{eg gl-strata}. If $N=2r+1$,  there are many strata of top dimension  $d-N$ 
(except in the trivial cases of $d=2r+1$ and $d=2r+2$). For example, the ${\mathfrak {sp}}_4$-stratification of $\sGr(5,8)$ 
has  four strata of dimension 3, see Section \ref{sec ex}. In all cases we have exactly one one-dimensional stratum corresponding to $n=1$, $\bs\La=(0)$, and $\bs k=(d-N)$.

\medskip

Essentially, we obtain the $\g_N$-stratification of $\s\Gr(N,d)$ by restricting the $\gl_N$-stratification of $\Gr(N,d)$ to $\s\Gr(N,d)$.

\medskip

For $X\in \sGr(N,d)$, the multiplicity of every zero of the Wronskian of $X$ is divisible by $r$ if $N=2r$ and by $N$ if $N=2r+1$. We define the reduced Wronski map $\overline\Wr: \sGr(N,d) \to \Gr(1, 2(d-N)+1)$ if $N=2r$ and $\overline\Wr: \sGr(N,d) \to \Gr(1, d-N+1)$ if $N=2r+1$ by sending $X$ to the $r$-th root of its Wronskian if $N=2r$ and to the $N$-th root
if $N=2r+1$. The $\g_N$-stratification of $\sGr(N,d)$ agrees with the reduced Wronski map and swallowtail $\gl_1$-stratification of
$\Gr(1, 2(d-N)+1)$ or $\Gr(1, d-N+1)$. For any $\gl_1$-stratum $\Omega_{\bs m}$ the preimage of $\Omega_{\bs m}$ under $\overline\Wr$ is the union
of $\g_N$-strata, see Proposition \ref{B preimage}, and the restriction of the reduced Wronski map to each of those
strata $\s\Omega_{\bLa,\bs k}$ is a ramified covering over $\Omega_{\bs m}$, see Proposition \ref{prop cov}.

\medskip

Our definition of the $\gl_N$-stratification is motivated by the connection to the Gaudin model of type $\rm A$, see Theorem \ref{thm bijection}. Similarly, our definition of the self-dual Grassmannian and of the $\g_N$-stratification is motivated by the connection to the Gaudin models of types $\rm B$ and $\rm C$, see Theorem \ref{bi rep sgr}.

It is interesting to study the geometry and topology of strata, cycles, and of self-dual Grassmannian, see Section \ref{sec ex}.

\medskip

The exposition of the material is as follows. 
In Section \ref{sec lie algs} we introduce the $\gl_N$ Bethe algebra. In Section \ref{sec gr} we describe the $\gl_N$-stratification 
of $\Gr(N,d)$.
In Section \ref{sec sgr} we define the $\g_N$-stratification of the self-dual Grassmannian $\sGr(N,d)$.
In  Section \ref{sec more notation} we recall the interrelations of the Lie algebras  $\sln_N$, $\mathfrak{so}_{2r+1}$, $\mathfrak{sp}_{2r}$.
In Section \ref{sec oper} we discuss $\g$-opers and their relations to self-dual spaces. Section \ref{sec proof} contains proofs of theorems formulated in Sections \ref{sec gr} and \ref{sec sgr}.
In Appendix A we describe the bijection between the self-dual spaces and the set of $\gl_N$ Bethe vectors fixed by the
Dynkin diagram automorphism of $\gl_N$.

\medskip

{\bf Acknowledgments.}
The authors thank V. Chari, A. Gabrielov, and L. Rybnikov for useful discussions. A.V. was supported in part by NSF grants DMS-1362924, 
DMS-1665239, and Simons Foundation grant \#336826. E.M. was supported in part by Simons Foundation grant \#353831.

\section{Lie algebras}\label{sec lie algs}
\subsection{Lie algebra $\gl_N$}
Let $e_{ij}$, $i,j=1,\dots,N$, be the standard generators of the Lie algebra $\gl_N$, satisfying the relations $[e_{ij},e_{sk}]=\delta_{js}e_{ik}-\delta_{ik}e_{sj}$. We identify the Lie algebra $\sln_N$ with the subalgebra of $\gl_N$ generated by the elements $e_{ii}-e_{jj}$ and $e_{ij}$ for $i\ne j$, $i,j=1,\dots,N$.
	
Let $M$ be a $\gl_N$-module. A vector $v\in M$ has weight $\la=(\la_1,\dots,\la_N)\in\C^N$ if $e_{ii}v=\la_iv$ for $i=1,\dots,N$. A vector $v$ is called \emph{singular} if $e_{ij}v=0$ for $1\lle i<j\lle N$.
	
We denote by $(M)_{\la}$ the subspace of $M$ of weight $\la$, by $(M)^\sing$ the subspace of $M$ of all singular vectors and by $(M)_\la^\sing$ the subspace of $M$ of all singular vectors of weight $\la$.
	
Denote by $V_{\la}$ the irreducible $\gl_N$-module with highest weight $\la$. 
	
The $\gl_N$-module $V_{(1,0,\dots,0)}$ is the standard $N$-dimensional vector representation of $\gl_N$, which we denote by $L$.
	
A sequence of integers $\la = (\la_1,\dots,\la_N)$ such that $\la_1\gge\la_2\gge\dots\gge\la_N\gge 0$ is called
\emph{a partition with at most $N$ parts}. Set $|\la|=\sum_{i=1}^N\la_i$. Then it is said that $\la$ is a partition of $|\la|$. The $\gl_N$-module $L^{\otimes n}$ contains the module $V_{\la}$ if and only if $\la$ is a partition of $n$ with at most $N$ parts.
	
Let $\la,\mu$ be partitions with at most $N$ parts. We write $\la\subseteq\mu$ if and only if $\la_i\lle \mu_i$ for $i=1,\dots,N$.

\subsection{Simple Lie algebras}\label{sec sla}
Let $\g$ be a simple Lie algebra over $\C$ with Cartan matrix $A=(a_{i,j})_{i,j=1}^r$. Let $D=\mathrm{diag}\{d_1,\dots,d_r\}$ be the diagonal matrix with positive relatively prime integers $d_i$ such that $DA$ is symmetric.
	
Let $\h\subset\g$ be the Cartan subalgebra and let $\g=\mathfrak n_-\oplus\h\oplus \mathfrak n_+$ be the Cartan decomposition. Fix simple roots $\alpha_1,\dots,\alpha_r$ in $\h^*$. Let $\calpha_1,\dots,\calpha_r\in \h$ be the corresponding coroots. Fix a nondegenerate invariant bilinear form $(,)$ in $\g$ such that $(\calpha_i,\calpha_j)=a_{i,j}/d_j$. The corresponding invariant bilinear form in $\h^*$ is given by $(\alpha_i,\alpha_j)=d_ia_{i,j}$.
We have $\langle \lambda,\calpha_i\rangle=2(\lambda,\alpha_i)/(\alpha_i,\alpha_i)$ for $\lambda\in\h^*$. In particular, $\langle\alpha_j,\calpha_i\rangle=a_{i,j}$. Let $\omega_1,\dots,\omega_r\in\h^*$ be the fundamental weights, $\langle \omega_j,\calpha_i\rangle=\delta_{i,j}$.
	
Let $\mathcal P=\{\la\in\mathfrak h^*|\langle \la,\calpha_i\rangle\in\mathbb Z,\ i=1,\dots,r\}$ and $\mathcal P^+=\{\la\in\mathfrak h^*|\langle \la,\calpha_i\rangle\in\mathbb Z_{\gge 0},\ i=1,\dots,r\}$ be the weight lattice and the cone of dominant integral weights.
	
For $\la\in {\mathfrak h}^*$, let $V_\la$ be the irreducible $\g$-module with highest weight $\la$. We denote $\langle \la,\calpha_i\rangle$ by $\lambda_i$ and sometimes write $(\la_1,\la_2,\dots,\la_r)$ for $\la$.
	
Let $M$ be a $\g$-module. Let $(M)^\mathrm{sing}=\{v\in M~|~\n_+v=0\}$ be the subspace of singular vectors in $M$. For $\mu\in\h^*$ let $(M)_\mu=\{v\in M~|~hv=\mu(h)v,\text{ for all }h\in\h\}$  be the subspace of $M$ of vectors of weight $\mu$. Let $(M)_\mu^\mathrm{sing}=M^\mathrm{sing}\cap (M)_\mu$ be the subspace of singular vectors in $M$ of weight $\mu$.
	
Given a $\g$-module $M$, denote by $(M)^{\g}$ the subspace of $\g$-invariants in $M$. The subspace $(M)^{\g}$ is the multiplicity space of the trivial $\g$-module in $M$. The following facts are well known. Let $\la$, $\mu$ be partitions with at most $N$ parts, $\dim(V_{\la}\otimes V_{\mu})^{\sln_N}=1$ if $\la_i=k-\mu_{N+1-i}$, $i=1,\dots,N$, for some integer $k\gge \mu_1$ and $0$ otherwise. Let $\la$, $\mu$ be $\g$-weights, $\dim(V_{\la}\otimes V_{\mu})^{\g}=\delta_{\la,\mu}$ for $\g=\mathfrak{so}_{2r+1},\mathfrak{sp}_{2r}$.
	
For any Lie algebra $\g$, denote by $\mc U(\g)$ the universal enveloping algebra of $\g$.
	
\subsection{Current algebra $\g[t]$}Let $\g[t] = \g\otimes\C[t]$ be the Lie algebra of $\g$-valued polynomials with the pointwise commutator. We call it the \emph{current algebra} of $\g$. We identify the Lie algebra $\g$ with the subalgebra $\g\otimes 1$ of constant polynomials in $\g[t]$. Hence, any $\g[t]$-module has the canonical structure of a $\g$-module.
	
The standard generators of $\gl_N[t]$ are $e_{ij}\otimes t^p$, $i,j=1,\dots,N$, $p\in \Z_{\gge 0}$. They satisfy the relations $[e_{ij}\otimes t^p,e_{sk}\otimes t^q]=\delta_{js}e_{ik}\otimes t^{p+q}-\delta_{ik}e_{sj}\otimes t^{p+q}$.
	
It is convenient to collect elements of $\g[t]$ in generating series of a formal variable $x$. For $g\in \g$, set
\beq\label{eq generating series}
g(x)=\sum_{s=0}^\infty (g\otimes t^s)x^{-s-1}.
\eeq
For $\gl_N[t]$ we have $(x_2-x_1)[e_{ij}(x_1),e_{sk}(x_2)]=\delta_{js}(e_{ik}(x_1)-e_{ik}(x_2))-\delta_{ik}(e_{sj}(x_1)-e_{sj}(x_2))$.
	
For each $a\in\C$, there exists an automorphism $\tau_a$ of $\g[t]$, $\tau_a:g(x)\to g(x-a)$. Given a $\g[t]$-module $M$, we denote by $M(a)$ the pull-back of $M$ through the automorphism $\tau_a$. As $\g$-modules, $M$ and $M(a)$ are isomorphic by the identity map.
	
We have the evaluation homomorphism, $\mathrm{ev}: \g[t]\to \g$, $\mathrm{ev} : g(x)\to g x^{-1}$. Its restriction to the subalgebra $\g\subset \g[t]$ is the identity map. For any $\g$-module $M$, we denote by the same letter the $\g[t]$-module, obtained by pulling $M$ back through the evaluation homomorphism. For each $a\in\C$, the $\g[t]$-module $M(a)$ is called an \emph{evaluation module}.
	
For $\g=\sln_N$, ${\mathfrak{sp}}_{2r}$, ${\mathfrak{so}}_{2r+1}$, it is well known that finite-dimensional irreducible $\g[t]$-modules are tensor products of evaluation modules $V_{\la^{(1)}}(z_1)\otimes\dots\otimes V_{\la^{(n)}}(z_n)$ with dominant integral $\g$-weights $\la^{(1)},\dots,\la^{(n)}$ and distinct evaluation parameters $z_1,\dots,z_n$.
	
\subsection{Bethe algebra}\label{sec: glN bethe}
Let $S_l$ be the permutation group of the set $\{1,\dots,l\}$. Given an $N \times N$ matrix $B$ with possibly noncommuting entries $b_{ij}$, we define its \emph{row determinant} to be
\[\rdet~B=\sum_{\sigma\in S_N}(-1)^{\sigma}b_{1\sigma(1)}b_{2\sigma(2)}\dots b_{N\sigma(N)}.\]
	
Define \emph{the universal differential operator} $\cDB$ by
	\beq\label{eq rdet}
	\cDB=\rdet(\delta_{ij}\pa_x-e_{ji}(x))_{i,j=1}^N.
	\eeq
	It is a differential operator in variable $x$, whose coefficients are formal power series in $x^{-1}$ with coefficients in $\mc U(\gl_N[t])$,
	\beq\label{eq rdet expanded}
	\cDB=\pa_x^N+\sum_{i=1}^N B_i(x)\pa_x^{N-i},
	\eeq
	where
	\[B_i(x)=\sum_{j=i}^{\infty}B_{ij}x^{-j}\]
	and $B_{ij}\in \mc U(\gl_N[t])$, $i=1,\dots,N$, $j\in\Z_{\gge i}$. We call the unital subalgebra of $\mc U(\gl_N[t])$
	generated by $B_{ij}\in \mc U(\gl_N[t])$, $i=1,\dots,N$, $j\in\Z_{\gge i}$, the \emph{Bethe algebra} of $\gl_N$ and denote it by $\cB$.
	
	The Bethe algebra $\cB$ is commutative and commutes with the subalgebra $\mc U(\gl_N)\subset \mc U(\gl_N[t])$, see \cite{T}. As a subalgebra of $\mc U(\gl_N[t])$, the algebra $\cB$ acts on any $\gl_N[t]$-module $M$. Since $\cB$ commutes with $\mc U(\gl_N)$, it preserves the subspace of singular vectors $(M)^\sing$ as well as weight subspaces of $M$. Therefore, the subspace $(M)^{\sing}_{\la}$ is $\cB$-invariant for any
	weight $\la$.
	
	\medskip
	
	We denote $M(\infty)$ the $\gl_N$-module $M$ with the trivial action of the Bethe algebra $\mc B$. More generally,
	for a $\gl_N[t]$-module $M'$, we denote by $M'\otimes M(\infty)$ the $\gl_N$-module where we define the action of $\mc B$ so that it acts trivially on $M(\infty)$. Namely, the element $b\in\mc B$ acts on $M'\otimes M(\infty)$ by $b\otimes 1$.
    
	Note that for $a\in \C$ and $\gl_N$-module $M$, the action of $e_{ij}(x)$ on $M(a)$ is given by $e_{ij}/(x-a)$ on $M$. Therefore, 
	the action of series $B_i(x)$ on the module  $M'\otimes M(\infty)$ is the limit of the action of the series $B_i(x)$ on the module $M'\otimes M(z)$ as $z\to \infty$ in the sense of rational functions of $x$. However, such a limit of the action of coefficients $B_{ij}$ on the module $M'\otimes M(z)$ as $z\to \infty$ does not exist.

	Let $M=V_\la$ be an irreducible $\gl_N$-module and let $M'$ be an irreducible finite-dimensional $\gl_N[t]$-module. Let $c$ be the value of the $\sum_{i=1}^Ne_{ii}$ action on $M'$.	
	
	\begin{lem}\label{nonzero weight}
	We have an isomorphism of vector spaces: 
	$$\pi:\ (M'\otimes V_{\la})^{\sln_N}\to (M')^{\sing}_{\bar\la}, \ {\rm where}\ \bar\la_i=\frac{c+|\la|}{N}-\la_{N+1-i},$$ 
	given by the projection to a lowest weight vector in $V_{\la}$.  
	The map $\pi$ is an isomorphism of $\mc B$-modules $(M'\otimes V_{\la}(\infty))^{\sln_N}\to (M')^{\sing}_{\bar\la}$. \qed
	\end{lem}
	
	\medskip
	
	Consider $\bP^1:=\C\cup\{\infty\}$. Set
	\[
	{\mathring{\bP}}_n:=\{\bm z=(z_1,\dots,z_n)\in (\mb{P}^1)^n~|~z_i\ne z_j\ \text{ for }\ 1\lle i<j\lle n\},
	\]
	\[
	\R{\mathring{\bP}}_n:=\{\bm z=(z_1,\dots,z_n)\in {\mathring{\bP}}_n~|~z_i\in\R \text{ or }z_i=\infty,\ \text{ for }\ 1\lle i\lle n\}.
	\]
	
We are interested in the action of the Bethe algebra $\cB$ on the tensor product $\bigotimes_{s=1}^n V_{ \la^{(s)}}(z_s)$, where
$\bLa=(\la^{(1)},\dots,\la^{(n)})$ is a sequence of partitions with at most $N$ parts and $\bm z=(z_1,\dots,z_n)\in {\mathring{\bP}}_n$. By Lemma \ref{nonzero weight}, it is sufficient to consider spaces of invariants $(\bigotimes_{s=1}^n V_{\la^{(s)}}(z_s))^{\sln_N}$. For brevity, we write $V_{\bLa,\bm z}$ for the $\mc B$-module $\bigotimes_{s=1}^n V_{ \la^{(s)}}(z_s)$ and $V_{\bLa}$ for the $\gl_N$-module $\bigotimes_{s=1}^n V_{\la^{(s)}}$.
	
	Let $v\in V_{\bLa,\bm z}$ be a common eigenvector of the Bethe algebra $\mc B$, $B_i(x)v = h_i(x)v$, $i=1,\dots,N$. Then	we call the scalar differential operator
	\[
	\mathcal D_v=\pa_x^N+\sum_{i=1}^Nh_i(x)\pa_x^{N-i}
	\]
	the \emph{differential operator associated with the eigenvector $v$.}
	
	\section{The $\gl_N$-stratification of Grassmannian}\label{sec gr}
	Let $N$, $d\in \Z_{>0}$ such that $N\lle d$.
	\subsection{Schubert cells}\label{sec schubert}
	Let $\C_d[x]$ be the space of polynomials in $x$ with complex coefficients of degree less than $d$. We have $\dim \C_d[x] = d$. Let $\mathrm{Gr}(N, d)$ be the Grassmannian of all $N$-dimensional subspaces in $\C_d[x]$. The Grassmannian $\Gr(N,d)$ is a
	smooth projective complex variety of dimension $N(d-N)$. 
	
	Let $\R_d[x]\subset\C_d[x]$ be the space of polynomials in $x$ with real coefficients of degree less than $d$. Let $\Gr^\R(N,d)\subset \Gr(N,d)$ be the set of subspaces which have a basis consisting of polynomials with real coefficients. For $X\in\Gr(N,d)$ we have 	$X\in\Gr^\R(N,d)$ if and only if $\dim_\R(X\cap \R_d[x])=N$. We call such points $X$ {\it real}.

	For a full flag $\mc F=\{0\subset F_1\subset F_2\subset \dots\subset F_d=\C_{d}[x]\}$ and a partition $\la=(\la_1,\dots,\la_N)$ such that $\la_1\lle d-N$, the Schubert cell $\Omega_{\la}(\mc F)\subset\Gr(N,d)$ is given by
	\begin{align*}
	\Omega_{\la}(\mc F)=\{X\in\Gr(N,d)~|~&\dim(X\cap F_{d-j-\la_{N-j}})=N-j,\\~&\dim(X\cap F_{d-j-\la_{N-j}-1})=N-j-1\}.
	\end{align*} 
	We have $\mathrm{codim}~\Omega_{\la}(\mc F)=|\la|$.
	
	The Schubert cell decomposition associated to a full flag $\mc F$, see for example \cite{GH}, is given by
	\beq\label{eq grass decom}
	\Gr(N,d)=\bigsqcup_{\la,\ \la_1\lle d-N}\Omega_{\la}(\mc F).
	\eeq
	The Schubert cycle  $\overline{\Omega}_{\la}(\mc F)$ is the closure of a Schubert cell $\Omega_{\la}(\mc F)$ in the Grassmannian $\Gr(N,d)$. Schubert cycles are algebraic sets with very rich geometry and topology.
	It is well known that Schubert cycle  $\overline{\Omega}_{\la}(\mc F)$  is described by the formula
	\beq\label{Omega bar}
	 \overline{\Omega}_{\la}(\mc F)=\bigsqcup_{\substack{\la\subseteq\mu,\\ \mu_1\lle d-N}}\Omega_{\mu}(\mc F).
	\eeq
	
	Given a partition $\la = (\la_1,\dots,\la_N)$ such that $\la_1\lle d-N$, introduce
	a new partition
	\[
	\bar\la=(d-N-\la_N,d-N-\la_{N-1},\dots,d-N-\la_1).
	\]
	We have $|\la|+|\bar\la|=N(d-N)$.
	
	Let $\mc F(\infty)$ be the full flag given by
	\begin{align}\label{F(inf)}
	\mc F(\infty)=\{0\subset \C_1[x]\subset\C_2[x]\subset\dots\subset\C_d[x]\}.
	\end{align}

	The subspace $X$ is a point of $\Omega_{\la}(\mc F(\infty))$ if and only if for every $i=1,\dots,N$, it
	contains a polynomial of degree $\bar\la_i+N-i$. 
	
	For $z\in\C$, consider the full flag
	\begin{align}\label{F(z)}
	\mc F(z)=\{0\subset (x-z)^{d-1}\C_1[x]\subset(x-z)^{d-2}\C_2[x]\subset\dots\subset\C_d[x]\}.
	\end{align}
	
	The subspace $X$ is a point of $\Omega_{\la}(\mc F(z))$ if and only if for every $i=1,\dots,N$, it
	contains a polynomial with a root at $z$ of order $\la_i+N-i$.
	
	A point $z\in\C$ is called a {\it base point} for a subspace $X\subset \C_d[x]$ if $g(z)=0$ for every $g\in X$.
	\subsection{Intersection of Schubert cells}
	Let $\bLa=(\la^{(1)},\dots,\la^{(n)})$ be a sequence of partitions with at most $N$ parts and $\bm z=(z_1,\dots,z_n)\in {\mathring{\bP}}_n$. Set $|\bLa|=\sum_{s=1}^n|\la^{(s)}|$. 
	
	The following lemma is elementary.

\begin{lem}
	If $\dim(V_\bLa)^{\sln_N}>0$, then $|\bLa|$ is divisible by $N$. Suppose further $|\bLa|=N(d-N)$, then $\la_1^{(s)}\lle d-N$ for $s=1,\dots,n$.\qed
\end{lem}
	
	Assuming $|\bLa|=N(d-N)$, denote by $\Omega_{\bLa,\bs z}$ the intersection of the Schubert cells:
	\beq\label{Omega}
	\Omega_{\bLa,\bs z}= \bigcap_{s=1}^n\Omega_{\la^{(s)}}(\mc F(z_s)).
	\eeq
	Note that due to our assumption, $\Omega_{\bLa,\bs z}$ is a finite subset of $\Gr(N,d)$. Note also that $\Omega_{\bLa,\bs z}$ is non-empty if and only if $\dim(V_\bLa)^{\sln_N}>0$.

	\begin{thm}\label{thm bijection}
		Suppose $\dim(V_\bLa)^{\sln_N}>0$. Let $v\in (V_{\bLa,\bm z})^{\sln_N}$ be an eigenvector of the Bethe algebra $\cB$. Then $\Ker\mathcal D_v\in\Omega_{\bLa,\bm z}$. Moreover, the assignment $\kappa:v\mapsto \Ker\mathcal D_v$ is a bijective correspondence between the set of eigenvectors of the Bethe algebra in $(V_{\bLa,\bm z})^{\sln_N}$ (considered up to multiplication by nonzero scalars) and the set $\Omega_{\bLa,\bm z}$.
	\end{thm}
	\begin{proof}
		The first statement is Theorem 4.1 in \cite{MTV2} and the second statement is Theorem 6.1 in \cite{MTV3}.
	\end{proof}
	
We also have the following lemma, see for example \cite{MTV}.

\begin{lem}\label{lem completeness} Let $\bm z$ be a generic point in ${\mathring{\bP}}_n$. Then the action of the Bethe algebra $\cB$ on $(V_{\bLa,\bm z})^{\sln_N}$ is diagonalizable. In particular, this statement holds for any sequence $\bm z\in \R{\mathring{\bP}}_n$.\qed
\end{lem}
	
\subsection{The $\gl_N$-stratification of $\Gr(N,d)$}
The following definition plays an important role in what follows.

Define a partial order $\gge$ on the set of sequences of partitions with at most $N$ parts as follows. Let $\bLa=(\la^{(1)},\dots,\la^{(n)})$, $\bm\Xi=(\xi^{(1)},\dots,\xi^{(m)})$ be two sequences of partitions with at most $N$ parts. We say that $\bLa\gge \bm\Xi$ if and only if there exists a partition $\{I_1,\dots,I_m\}$ of the set $\{1,2,\dots,n\}$ such that
	\[
	{ \rm {Hom}}_{\gl_N} (V_{\xi^{(i)}}, \bigotimes_{j\in I_i}V_{\la^{(j) }}\big)\neq 0,\quad i=1,\dots,m.
	\]Note that $\bLa$ and $\bm\Xi$ are comparable only if $|\bLa|=|\bm{\Xi}|$.

We say that $\bLa=(\la^{(1)},\dots,\la^{(n)})$ is \emph{nontrivial} if and only if $(V_{\bLa})^{\sln_N}\ne 0$ and $|\la^{(s)}|> 0$, $s=1,\dots,n$. The sequence $\bLa$ will be called $d$-\emph{nontrivial} if $\bLa$ is nontrivial and $|\bLa|=N(d-N)$. 
	
Suppose $\bm\Xi$ is $d$-nontrivial. If $\bLa\gge \bm\Xi$ and $|\la^{(s)}|>0$ for all $s=1,\dots,n$, then $\bLa$ is also $d$-nontrivial.
	
Recall that $\Omega_{\bLa,\bm z}$ is the intersection of Schubert cells for each given $\bm z$, see \eqref{Omega}, define $\Omega_{\bLa}$ by the formula
\beq\label{strata}
\Omega_{\bLa}:=\bigcup_{\bm z\in {\mathring{\bP}}_n}\Omega_{\bLa,\bm z}\subset \Gr(N,d).
\eeq
By definition, $\Omega_{\bLa}$ does not depend on the order of $\la^{(s)}$ in the sequence $\bLa=(\la^{(1)},\dots,\la^{(n)})$.
	Note that $\Omega_{\bLa}$ is a constructible subset of the Grassmannian
	$\Gr(N,d)$  in Zariski topology. We call  $\Omega_{\bLa}$ with a $d$-nontrivial $\bLa$ a {\it $\gl_N$-stratum} of $\Gr(N,d)$.
	
	 Let $\mu^{(1)},\dots,\mu^{(a)}$ be the list of all distinct partitions in $\bLa$. Let $n_i$ be the number of occurrences of ${\mu}^{(i)}$ in $\bLa$, $i=1,\dots,a$, then $\sum_{i=1}^an_i=n$. Denote $\bm n=(n_1,\dots,n_a)$. We shall write $\bLa$ in the following order: $\la^{(i)}=\mu^{(j)}$ for $\sum_{s=1}^{j-1}n_s+1\lle i\lle \sum_{s=1}^{j}n_s$, $j=1,\dots,a$.
	
	 Let $S_{\bm n;n_i}$ be the subgroup of the symmetric group $S_n$ permuting $\{n_1+\cdots+n_{i-1}+1,\dots,n_1+\cdots+n_{i}\}$, $i=1,\dots,a$. Then the group $S_{\bm n}=S_{\bm n;n_1}\times S_{\bm n;n_2}\times \dots\times S_{\bm n;n_a}$ acts freely on ${\mathring{\bP}}_n$ and on $\R{\mathring{\bP}}_n$. Denote by ${\mathring{\bP}}_n/S_{\bm n}$ and $\R{\mathring{\bP}}_n/S_{\bm n}$ the sets of orbits. 
	
	\begin{prop}\label{prop deg of cover A}
		Suppose $\bLa=(\la^{(1)},\dots,\la^{(n)})$ is $d$-nontrivial. The stratum $\Omega_{\bLa}$ is a ramified covering of ${\mathring{\bP}}_n/S_{\bm n}$. Moreover, the degree of the covering is equal to $\dim (V_{\bLa})^{\sln_N}$.  In particular, $\dim\Omega_{\bLa}=n$. Over $\R{\mathring{\bP}}_n/S_{\bm n}$, this covering is unramified of the same degree, moreover all points in fibers are real. 
	\end{prop}
	\begin{proof}
		The statement follows from Theorem \ref{thm bijection}, Lemma \ref{lem completeness}, and Theorem 1.1 of \cite{MTV2}.
	\end{proof}
	
	Clearly, we have the following theorem.
	\begin{thm}\label{thm gr dec}
		We have
		\beq\label{eq decom gr}
		\Gr(N,d)=\bigsqcup_{d\text{-nontrivial }\bLa}\Omega_{\bLa}.
		\eeq \qed
	\end{thm}
	
	Next, for a $d$-nontrivial $\bLa$, 
	we call the closure of $\Omega_{\bLa}$ inside $\Gr(N,d)$, a {\it $\gl_N$-cycle}. The $\gl_N$-cycle $\overline{\Omega}_{\bLa}$ is an algebraic set. We describe the $\gl_N$-cycles as unions of $\gl_N$-strata.
	
	Let $\bLa=(\la^{(1)},\dots,\la^{(n)})$ and $\bm\Xi=(\xi^{(1)},\dots,\xi^{(n-1)})$ be such that $\bm\Xi\lle \bLa$. We call $\Omega_{\bm\Xi}$ a \emph{simple degeneration} of $\Omega_{\bLa}$ if and only if both $\bLa$ and $\bm\Xi$ are $d$-nontrivial. In view of Theorem \ref{thm bijection}, taking a simple degeneration is equivalent to making two coordinates of $\bm z$ collide.
	
	\begin{thm}\label{thm simple deg A}
		If $\Omega_{\bm\Xi}$ is a simple degeneration of $\Omega_{\bLa}$, then $\Omega_{\bm\Xi}$ is contained in the $\gl_N$-cycle $\overline{\Omega}_{\bLa}$.
	\end{thm}
	Theorem \ref{thm simple deg A} is proved in Section \ref{sec proof A}.
	
	\medskip
	
	Suppose $\bm\Theta=(\theta^{(1)},\dots,\theta^{(l)})$ is $d$-nontrivial and $\bLa\gge \bm\Theta$. Then, it is clear that 
	$\Omega_{\bm\Theta}$ is obtained from $\Omega_{\bLa}$ by a sequence of simple degenerations. We call $\Omega_{\bm\Theta}$ a \emph{degeneration} of $\Omega_{\bLa}$.
	
	\begin{cor}
		If $\Omega_{\bm\Theta}$ is a degeneration of $\Omega_{\bLa}$, then $\Omega_{\bm\Theta}$ is contained in the $\gl_N$-cycle 
		$\overline{\Omega}_{\bLa}$. \qed
	\end{cor}
	
	\begin{thm}\label{thm A strata}
		For $d$-nontrivial $\bLa$, we have
		\beq\label{eq A strata}
		\overline{\Omega}_{\bLa}=\bigsqcup_{\substack{\bm\Xi\lle \bLa,\\ d\text{-nontrivial }\bm\Xi}}\Omega_{\bm\Xi}.
		\eeq
	\end{thm}
	Theorem \ref{thm A strata} is proved in Section \ref{sec proof A}.
	
	\medskip 
	
		Theorems \ref{thm gr dec} and \ref{thm A strata} imply that the subsets $\Omega_{\bLa}$ with $d$-nontrivial $\bLa$ give a stratification of $\Gr(N,d)$. We call it the {\it $\gl_N$-stratification of $\Gr(N,d)$}.
	 
	\begin{eg}\label{A ex} We give an example of the $\gl_2$-stratification for $\Gr(2,4)$ in the following picture. In the picture, we simply write $\bLa$ for $\Omega_{\bLa}$. We also write tuples of numbers with bold font for $4$-nontrivial tuples of partitions, solid arrows for simple degenerations between $4$-nontrivial tuples of partitions. The dashed arrows go between comparable sequences where the set $\Omega_{\bm\Xi}$ corresponding to the smaller sequence is empty.

		
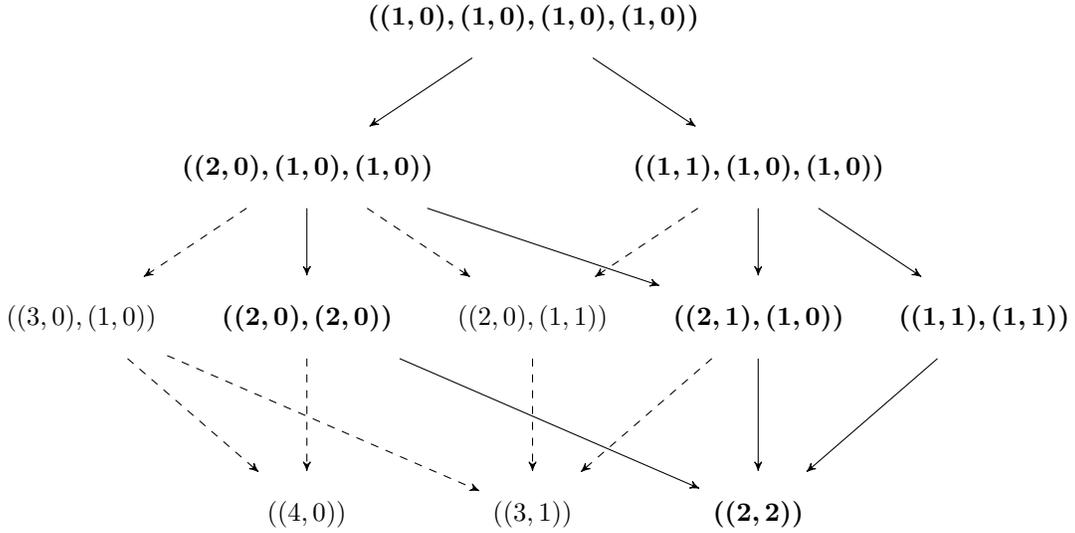
\begin{center}
			\begin{tikzpicture}[->,>=stealth',shorten >=1pt,auto,node distance=2.8cm]
			\tikzstyle{every state}=[rectangle,draw=none,text=black]
			
			\node[state]         (S) at (-4, 4.6)              {$\bm{((1,0),(1,0),(1,0),(1,0))}$};
			\node[state]         (xin1) at (-7, 2.6)           {$\bm{((2,0),(1,0),(1,0))}$};
			\node[state]         (xin2) at (-1, 2.6)        {$\bm{((1,1),(1,0),(1,0))}$};
			\node[state]         (xout1) at (-1,0.6 )          {$\bm{((2,1),(1,0))}$};
			\node[state]         (xout2) at (-7, 0.6)        {$\bm{((2,0),(2,0))}$};
			\node[state]         (xout7) at (2, 0.6)        {$\bm{((1,1),(1,1))}$};
			\node[state]         (xout6) at (-10, 0.6)        {$((3,0),(1,0))$};
			\node[state]         (xout8) at (-4, 0.6)        {$((2,0),(1,1))$};
			\node[state]         (DC1) at (-4, -2)           {$((3,1))$};
			\node[state]         (DC) at (-1, -2)           {$\bm{((2,2))}$};
			\node[state]         (DC2) at (-7, -2)           {$((4,0))$};
			
			\path
			(S) edge                  node[xshift=-1.9cm,yshift=0.2cm]{} (xin1)
			edge              node[xshift=0.4cm,yshift=-0.4cm]{} (xin2)	
			(xin1) edge  node[xshift=-1.6cm,yshift=0.30cm]{} (xout2)
			(xin1) edge  node[xshift=-1.6cm,yshift=-0.23cm]{} (xout1)
			(xin1) edge[dashed]  (xout6)
			(xin2) edge  node[xshift=0.15cm,yshift=0.3cm]{} (xout1)
			
			(xin2) edge  node[xshift=-0.0cm,yshift=-0.2cm]{} (xout7)
			(xout1) edge   node[xshift=0cm,yshift=0cm]{} (DC)
			(xout1) edge[dashed] (DC1)
			(xin1) edge[dashed] (xout8)
			(xin2) edge[dashed] (xout8)
			(xout8) edge[dashed] (DC1)
			(xout2) edge[dashed] (DC2)
			(xout6) edge[dashed] (DC1)
			(xout6) edge[dashed] (DC2)
			(xout2) edge (DC)
			(xout7) edge (DC);
			\end{tikzpicture}
		\end{center}
		In particular, $\Omega_{((1,0),(1,0),(1,0),(1,0))}$ is dense in $\Gr(2,4)$.
	\end{eg}

	\begin{rem}
		In general, for $\Gr(N,d)$, let $\epsilon_1=(1,0,\dots,0)$ and let
		\[
		\bLa=(\underbrace{\epsilon_1,\epsilon_1,\dots,\epsilon_1}_{N(d-N)}).
		\]
		Then $\bLa$ is $d$-nontrivial, and $\Omega_{\bLa}$ is dense in $\Gr(N,d)$. Clearly,  $\Omega_{\bLa}$ consists of spaces of polynomials whose Wronskian (see Section \ref{sec wronski gr}) has only simple roots.
	\end{rem}

	\begin{rem}\label{translations}
	The group of affine translations acts on $\C_d[x]$ by changes of variable. Namely, for  $a\in\C^*,b\in\C$, we have a map sending
	$f(x)\mapsto f(ax+b)$ for all $f(x)\in\C_d[x]$. This group action preserves the Grassmannian $\Gr(N,d)$ and the strata $\Omega_{\bLa}$.
	\end{rem}
	
	\subsection{The case of $N=1$ and the Wronski map}\label{sec wronski gr} We show that the decomposition in Theorems \ref{thm gr dec} and \ref{thm A strata} respects the Wronski map.
	
	From now on, we use the convention that $x-z_s$ is considered as the constant function $1$ if $z_s=\infty$.
	
	Consider the Grassmannian of lines $\Gr(1,\tl d)$.
	By Theorem \ref{thm gr dec}, the decomposition of $\Gr(1,\tl d)$ is parameterized by unordered sequences of positive integers $\bm m=(m_1,\dots,m_n)$ such that $|\bm m|=\tl d-1$. 
	
	Let $\bs z=(z_1,\dots,z_n)\in \mathring{\bP}_n$. We have $\C f\in \Omega_{\bm m,{\bm z}}$ if and only if
	\[
	f(x)=a\prod_{s=1}^n(x-z_s)^{m_s}, \quad a\neq 0.\]
	In other words, the stratum $\Omega_{\bm m}$ of the $\gl_1$-stratification \eqref{eq decom gr} of $\Gr(1,\tl d)$ consists of all points in $\Gr(1,\tl d)$ whose representative polynomials have $n$ distinct roots (one of them can be $\infty$) of multiplicities $m_1,\dots,m_n$. 
	
	Therefore the $\gl_1$-stratification is exactly the celebrated swallowtail stratification.
	
	For $g_1(x),\dots,g_l(x)\in \C[x]$, denote by $\Wr(g_1(x),\dots,g_l(x))$ the \emph{Wronskian},
	\[
	\Wr(g_1(x),\dots,g_l(x))=\det(d^{i-1}g_j/dx^{i-1})_{i,j=1}^l.
	\]
	Let $X\in\Gr(N, d)$. The Wronskians of two bases of $X$ differ by a multiplication by a nonzero number. We call the monic polynomial representing the Wronskian the \emph{Wronskian} of $X$ and denote it by $\Wr(X)$. It is clear that $\deg_x \Wr(X)\lle N(d-N)$.
	
	The \emph{Wronski} map
	\[
	\Wr:\Gr(N,d)\to\Gr(1,N(d-N)+1)
	\]
	is sending $X\in \Gr(N,d)$ to $\C\Wr(X)$.
	
	The Wronski map is a finite algebraic map, see for example Propositions 3.1 and 4.2 in \cite{MTV7}, of degree $\dim(L^{\otimes N(d-N)})^{\sln_n}$, which is explicitly given by
	$$
	(N(d-N))!\ \frac{0!\ 1!\ 2!\dots (d-N-1)!}{N!\ (N+1)!\ (N+2)!\dots (d-1)!},
	$$
	see \cite{S}.
	
	Let $\bLa=(\la^{(1)},\dots,\la^{(n)})$ be $d$-nontrivial and $\bs z=(z_1,\dots,z_n)\in \mathring{\bP}_n$. If $X\in \Omega_{\bLa,\bm z}$, then one has
	\[
	\Wr(X)=\prod_{s=1}^n(x-z_s)^{|\la^{(s)}|}.
	\]
	
	Set $\tl d=N(d-N)+1$. Therefore, we have the following proposition.
	
	\begin{prop}\label{preimage}
	The preimage of the stratum $\Omega_{\bm m}$ of $\Gr(1,N(d-N)+1)$ under the Wronski map is a union of all $d$-nontrivial strata $\Omega_{\bLa}$ of $\Gr(N,d)$ such that $|\la^{(s)}|=m_s$, $s=1,\dots,n$. \qed
	\end{prop}
	
Let  $\bLa=(\la^{(1)},\dots,\la^{(n)})$ be an unordered sequence of partitions with at most $N$ parts. Let $a$ be the number of distinct partitions in $\bLa$. We can assume that $\la^{(1)},\dots,\la^{(a)}$ are all distinct 
and let $n_1,\dots,n_a$ be their multiplicities in $\bLa$, $n_1+\dots+n_a=n$. Define the {\it symmetry coefficient} of $\bLa $ as the product of multinomial coefficients:
\beq
\label{sym co}
b(\bLa) = \prod_i \frac{\left(\sum_{s=1,\dots,a, \ |\la^{(s)}|=i} n_s\right)! }  {\prod_{s=1,\dots,a,\ |\la^{(s)}|=i} (n_s)!}.
\eeq

\begin{prop}
\label{prop cov}
Let $\bLa=(\la^{(1)},\dots,\la^{(n)})$ be $d$-nontrivial. Then the Wronski map 
$\Wr|_{\Omega_\bLa} : \Omega_\bLa\to \Omega_{\bs m}$ is a ramified covering
 of degree $b(\bLa) \dim (V_{\bLa})^{\sln_N}$.  
	\end{prop}

	\begin{proof}
		The statement follows from Theorem \ref{thm bijection}, Lemma \ref{lem completeness}, and Proposition \ref{preimage}.
	\end{proof}
	
In other words, the $\mathfrak{gl}_N$-stratification of $\Gr(N,d)$ given by Theorems \ref{thm gr dec} and \ref{thm A strata},
is adjacent to the swallowtail $\gl_1$-stratification of $\Gr(1, N(d-N)+1)$ and the Wronski map.

\section{The $\g_N$-stratification of self-dual Grassmannian}\label{sec sgr}
	
It is convenient to use the notation:  $\g_{2r+1}=\mathfrak{sp}_{2r}$, and $\g_{2r}=\mathfrak{so}_{2r+1}$, $r\gge 2$. We also set $\g_3=\sln_2$. The case of	$\g_3=\sln_2$ is discussed in detail in Section \ref{sec proof N=3}. 
	
\subsection{Self-dual spaces}\label{sec self-dual}Let $\bLa=( \la^{(1)},\dots,\la^{(n)})$ be a tuple of partitions with at most $N$ parts such that $|\bLa|=N(d-N)$ and let $\bs z=(z_1,\dots,z_n)\in \mathring{\bP}_n$.
	
Define a tuple of polynomials $\bs T=(T_1,\dots,T_N)$ by
\beq\label{eq def T}
T_i(x)=\prod_{s=1}^n(x-z_s)^{\la_{i}^{(s)}-\la_{i+1}^{(s)}},\quad i=1,\dots,N,
\eeq
where $\la_{N+1}^{(s)}=0$. We say that $\bs T$ is {\it associated with} $\bLa,\bm z$.
	
Let $X\in \Omega_{\bLa,\bm z}$ and $g_1,\dots,g_i\in X$. Define the \emph{divided Wronskian $\Wr^\dag$ with respect to $\bLa,\bm z$} by
\[
\Wr^\dag(g_1,\dots,g_i)=\Wr(g_1,\dots,g_i)\prod_{j=1}^i T_{N+1-j}^{j-i-1},\quad i=1,\dots,N.
\]
Note that $\Wr^\dag(g_1,\dots,g_i)$ is a polynomial in $x$.
	
Given $X\in \Gr(N,d)$, define the \emph{dual space} $X^\dag$ of $X$ by
$$ X^\dag=\{ \Wr^\dag (g_1,\dots,g_{N-1})\ | \ g_i\in X,\ i=1,\dots,N-1\}. $$
	
\begin{lem}\label{dual exponents} If $X\in \Omega_{\bLa,\bm z}$, then $X^\dag\in \Omega_{\tl{\bLa},\bm z}\subset \Gr(N,\tl d)$, where
$$\tl d=\sum_{s=1}^n\la_1^{(s)}-d+2N,$$
and
$\tl{\bLa}=(\tl{\la}^{(1)},\dots,\tl{\la}^{(n)})$ is a sequence of partitions with at most $N$ parts such that
\[
\tl{\la}_i^{(s)}=\la_1^{(s)}-\la_{N+1-i}^{(s)},\quad i=1,\dots,N,\quad s=1,\dots,n.
\]\qed
\end{lem}
	
	Note that we always have $\tl{\la}^{(s)}_N=0$ for every $s=1,\dots,n$, hence $X^\dag$ has no base points.	
	
	Given a space of polynomials $X$ and a rational function $g$ in $x$, denote by $g\cdot X$ the space of rational functions of the form $g\cdot f$ with $f\in X$.
	
	A self-dual space is called a {\it pure self-dual space} if $X=X^\dag$. A space of polynomials $X$ is called \emph{self-dual} if $X = g\cdot X^\dag$ for some polynomial $g\in\C[x]$. In particular, if $X\in \Omega_{\bLa,\bm z}$ is self-dual, then $X = T_N\cdot X^\dag$, where $T_N$ is defined in \eqref{eq def T}. Note also, that if $X$ is self-dual then $g\cdot X$ is also self-dual.
	
	It is obvious that every point in $\Gr(2,d)$ is a self-dual space.
	
	 Let $\sGr(N,d)$ be the set of all self-dual spaces in $\Gr(N,d)$. We call $\sGr(N,d)$ the \emph{self-dual Grassmannian}.
	The self-dual Grassmannian $\sGr(N,d)$ is an algebraic subset of $\Gr(N,d)$. 
	
	Let  $\Omega_{\bLa,\bm z}$ be the finite set defined in \eqref{Omega}      
	and $\Omega_{\bLa}$ the set defined in \eqref{strata}. Denote by $\s\Omega_{\bLa,\bm z}$ the set of all self-dual spaces in $\Omega_{\bLa,\bm z}$ and by $\s\Omega_{\bLa}$ the set  of all self-dual spaces in $\Omega_{\bLa}$: 
	\begin{align*}
	\s\Omega_{\bLa,\bm z}=\Omega_{\bLa,\bm z}\bigcap \sGr(N,d) \quad {\rm and} \quad \s\Omega_{\bLa}=\Omega_{\bLa}\bigcap \sGr(N,d).
	\end{align*}
	We call the sets $\s\Omega_{\bLa}$ {\it $\g_N$-strata} of the self-dual Grassmannian.
 A stratum $\s\Omega_{\bLa}$ does not depend on the order of the set of partitions $\bLa$. Note that each $\s\Omega_{\bLa}$ is a constructible subset of the Grassmannian $\Gr(N,d)$ in Zariski
 topology.
 
A partition $\la$ with at most $N$ parts is called {\it $N$-symmetric} if $\la_{i}-\la_{i+1}=\la_{N-i}-\la_{N-i+1}$, $i=1,\dots, N-1$. If the stratum $\s\Omega_{\bLa}$ is nonempty, then all partitions $\la^{(s)}$ are $N$-symmetric, see also Lemma \ref{lem sym weight new} below.
	
	The self-dual Grassmannian is related to the Gaudin model in types B and C, see \cite{MV} and Theorem \ref{bi rep sgr} below. We show that $\sGr(N,d)$ also has a remarkable stratification structure similar to the $\gl_N$-stratification of $\Gr(N,d)$, governed by representation theory of $\g_N$, see Theorems \ref{thm sgr dec} and \ref{thm BC strata}.

	\begin{rem}\label{Schubert BC}
	The self-dual Grassmannian also has a stratification induced from the usual Schubert cell decomposition \eqref{eq grass decom}, \eqref{Omega bar}. 	For $z\in{\Bbb P}^1$, and an $N$-symmetric partition $\la$ with $\la_1\lle d-N$, set $\s\Omega_{\la}(\mc F(z))=\Omega_{\la}(\mc F(z))\cap \s\Gr(N,d)$. Then it is easy to see that
	\begin{align*}
	\s\Gr(N,d)=\underset{\underset {\mu_1\lle d-N}{N-{\rm symmetric\ \mu,}}} {\bigsqcup} \s\Omega_\mu(\mc F(z)) \quad {\rm and} \quad \overline{\s\Omega}_\la(\mc F(z))= \underset{\underset {\mu_1\lle d-N, \ \la\subseteq\mu}{N-{\rm symmetric}\ \mu,} } {\bigsqcup} \s\Omega_{\mu}(\mc F(z)).
	\end{align*}
	
	\end{rem}
	
	\subsection{Bethe algebras of types $\rm B$ and $\rm C$ and self-dual Grassmannian}\label{sec self-dual structure}
	
The Bethe algebra $\mc B$ (the algebra of higher Gaudin Hamiltonians) for a simple Lie algebras $\g$ were described in \cite{FFRe}. The Bethe algebra $\cB$ is a commutative subalgebra of $\mc U(\g[t])$ which commutes with the subalgebra $\mc U(\g)\subset \mc U(\g[t])$.
An explicit set of generators of the Bethe algebra in Lie algebras of types $\rm B$, $\rm C$, and $\rm D$ was given in \cite{M}.  Such a description in the case of $\gl_N$ is given above in Section \ref{sec: glN bethe}. For the case of $\g_N$ we only need the following fact. 
	
	Recall our notation $g(x)$ for the current of  $g\in\g$, see \eqref{eq generating series}.
	\begin{prop}[\cite{FFRe, M}]\label{prop bethe BC} Let $N>3$.
	There 
		exist elements $F_{ij}\in\g_N$, $i,j=1,\dots,N$, and polynomials $ G_s(x)$ in $d^k F_{ij}(x)/dx^k$, $s=1,\dots,N$, $k=0,\dots, N$,  such that the Bethe algebra of $\g_N$ is generated by coefficients of $G_s(x)$ considered as formal power series in $x^{-1}$.\qed
	\end{prop}
	
Similar to the $\gl_N$ case, for a collection of dominant integral $\g_N$-weights $\bLa=(\la^{(1)},\dots,\la^{(n)})$ and $\bs z=(z_1,\dots,z_n)\in \mathring{\bP}_n$, we set $V_{\bLa,\bm z}=\bigotimes_{s=1}^n V_{\la^{(s)}}(z_s)$, considered as a $\mc B$-module. 
Namely, if $\bs z\in\C^n$, then $V_{\bLa,\bm z}$ is a tensor product of evaluation $\g_N[t]$-modules and therefore a $\mc B$-module. 
If, say, $z_n=\infty$, then $\mc B$ acts trivially on $V_{\la^{(n)}}(\infty)$. More precisely, in this case, $b\in \mc B$ acts by $b\otimes 1$ where the first factor acts on   $\bigotimes_{s=1}^{n-1} V_{\la^{(s)}}(z_s)$ and $1$ acts on $V_{\la^{(n)}}(\infty)$.

We also denote $V_{\bLa}$ the module $V_{\bLa,\bm z}$ considered as a $\g_N$-module. 
	
Let $\mu$ be a dominant integral $\g_N$-weight and $k\in\Z_{\gge 0}$. Define an $N$-symmetric
partition $\mu_{A,k}$ with at most $N$ parts by the rule: $(\mu_{A,k})_N=k$ and
\beq\label{eq partition BC}
(\mu_{A,k})_i-(\mu_{A,k})_{i+1}=\begin{cases}
\langle \mu,\calpha_i\rangle,\quad &\text{if }1\lle i\lle \big[\frac{N}{2}\big],\\
\langle \mu,\calpha_{N-i}\rangle,\quad &\text{if }\big[\frac{N}{2}\big]< i\lle N-1.
\end{cases}
\eeq
We call $\mu_{A,k}$ the partition \emph{associated with weight $\mu$ and integer $k$}.
	
Let $\bLa=(\la^{(1)},\dots,\la^{(n)})$ be a sequence of dominant integral $\g_N$-weights and let $\bm k=(k_1,\dots,k_n)$ be an $n$-tuple of nonnegative integers. Then denote $\bLa_{A,\bm k}=(\la_{A,k_1}^{(1)},\dots,\la_{A,k_n}^{(n)})$ the sequence of partitions associated with $\la^{(s)}$ and $k_s$, $s=1,\dots,n$. 
	
We use notation $\mu_{A}=\mu_{A,0}$ and $\bLa_{A}=\bLa_{A,(0,\dots,0)}$.
	
\begin{lem}\label{lem sym weight new}
If $\bm\Xi$ is a $d$-nontrivial sequence of partitions with at most $N$ parts and $\s\Omega_{\bm\Xi}$ is nonempty, then $\bm\Xi$ has the form 
$\bm\Xi=\bLa_{A,\bm k}$ for 
a sequence of dominant integral $\g_N$-weights $\bLa=(\la^{(1)},\dots,\la^{(n)})$ and an $n$-tuple $\bm k$ of nonnegative integers. The pair $(\bLa,\bm k)$ is uniquely determined by $\bm \Xi$. Moreover, if $N=2r$, then $\sum_{s=1}^n\langle \la^{(s)},\calpha_r\rangle$ is even.
\end{lem}
	
\begin{proof}
	The first statement follows from Lemma \ref{dual exponents}. If $N=2r$ is even, the second statement follows from the equality
	\[
	N(d-N)=|\bm{\Xi}|=\sum_{s=1}^nr\Big(2\sum_{i=1}^{r-1}\langle \la^{(s)},\calpha_i\rangle+\langle \la^{(s)},\calpha_r\rangle\Big)+N\sum_{s=1}^nk_s.\qedhere
	\]
\end{proof}
Therefore the strata are effectively parameterized by sequences of dominant integral $\g_N$-weights and tuples of nonnegative integers. 
In what follows we write $\s\Omega_{\bLa,\bm k}$ for $\s\Omega_{\bLa_{A,\bm k}}$ and $\s\Omega_{\bLa,\bm k,\bm z}$ for $\s\Omega_{\bLa_{A,\bm k},\bm z}$.	

\medskip
	
	Define a formal differential operator
	$$
	\mc D^{\mc B}=\partial_x^N+\sum_{i=1}^NG_i(x)\pa_x^{N-i}.
	$$
	For a $\mc B$-eigenvector $v\in V_{\bLa,\bs z}$, $G_i(x)v=h_i(x)v$, 
	we denote $\mc D_v=\partial_x^N+\sum_{i=1}^Nh_i(x)\pa_x^{N-i}$ the corresponding scalar differential operator.
 	
	\begin{thm}\label{bi rep sgr} Let $N>3$.
	
	There exists a choice of generators $G_i(x)$ of the $\g_N$ Bethe algebra $\mc B$ (see Proposition \ref{prop bethe BC}), such that for any sequence of dominant integral $\g_N$-weights $\bLa=(\la^{(1)},\dots,\la^{(n)})$, any $\bm z\in \mathring{\bP}_n$, and any $\mc B$-eigenvector $v\in (V_{\bLa,\bm z})^{\g_N}$, we have ${\rm{Ker}}\ ((T_1\dots T_{N})^{1/2}\cdot\mc D_v \cdot (T_1\dots T_{N})^{-1/2})\in \s\Omega_{\bLa_{A},\bs z}$, where $\bm T=(T_1,\dots,T_N)$ is associated with $\bLa_A,\bm z$.
	
	Moreover, if $|\bLa_A|=N(d-N)$, then this defines a bijection between the joint eigenvalues of $\mc B$ on $(V_{\bLa,\bm z})^{\g_N}$ and $\s\Omega_{\bLa_{A},\bs z}\subset\Gr(N,d)$.
	\end{thm}
	\begin{proof}
		Theorem \ref{bi rep sgr} is deduced from \cite{R} in Section \ref{sec proof BC}.
	\end{proof}
	The second part of the theorem also holds for $N=3$, see Section \ref{sec proof N=3}.
	
	\begin{rem}\label{rem}
		In particular, Theorem \ref{bi rep sgr} implies that if $\dim(V_\bLa)^{\g_N}>0$, then $\dim(V_{\bLa_{A,\bm k}})^{\sln_N}>0$. 
		This statement also follows from Lemma \ref{lem dimension of singular space} given in the Appendix.
	\end{rem}
	
	We also have the following lemma from \cite{R}.
	
	\begin{lem}\label{lem BC completeness} Let $\bm z$  be a generic point in ${\mathring{\bP}}_n$. Then the action of the $\g_N$ Bethe algebra on $(V_{\bLa,\bm z})^{\g_N}$ is diagonalizable and has simple spectrum. In particular, this statement holds for any sequence $\bm z\in \R{\mathring{\bP}}_n$.\qed
	\end{lem}
	
\subsection{Properties of the strata}\label{sec prop strata} We describe simple properties of the strata  $\s\Omega_{\bLa,\bm k}$.
	
	Given $\bLa,\bs k,\bs z$, define $\tilde{\bLa},\tilde {\bm k}, \tilde {\bs z}$ by removing all zero components, 
	that is the ones with both $\la^{(s)}=0$ and $k_s=0$. 
	Then $\s\Omega_{\tl{\bLa},\tilde {\bs k},\tl{\bm z}}=\s\Omega_{\bLa,\bs k,\bm z}$ and  $\s\Omega_{\tl{\bLa},\tilde {\bs k}}=\s\Omega_{\bLa,\bs k}$. 	Also, by Remark \ref{rem}, if $(V_\bLa)^{\g_N} \ne 0$, then $\dim(V_{\bLa_{A,\bm k}})^{\sln_N}>0$, thus $|\bLa_{A,\bm k}|$ is divisible by $N$.

We say that $(\bLa,\bm k)$ is $d$-\emph{nontrivial} if and only if $(V_\bLa)^{\g_N} \ne 0$, $|\la_{A,k_s}^{(s)}|>0$, $s=1,\dots,n$, and $|\bLa_{A,\bm k}|=N(d-N)$. 
	
If $(\bLa,\bm k)$ is $d$-nontrivial then the corresponding stratum $\s\Omega_{\bs \La,\bm k}\subset \s\Gr(N,d)$ is nonempty, see Proposition \ref{prop bij BC} below.

\medskip
	
	Note that $|\bLa_{A,\bm k}|=|\bLa_{A}|+N|\bs k|$, where $|\bs k|=k_1+\dots +k_n$. In particular, if $(\bLa,\bs 0)$ is $d$-nontrivial then $(\bLa,\bs k)$ is $(d+|\bs k|)$-nontrivial.
	Further, there exists a bijection between $\Omega_{\bLa_A,\bm z}$ in $\Gr(N,d)$ and $\Omega_{\bLa_{A,\bm k},\bm z}$ in $\Gr(N,d+|\bs k|)$ given by
	\beq\label{mult}
	\Omega_{\bLa_A,\bm z}\rightarrow \Omega_{\bLa_{A,\bm k},\bm z},\quad X \mapsto \prod_{s=1}^n(x-z_s)^{k_s}\cdot X.
	\eeq
	Moreover, \eqref{mult} restricts to a bijection of $\s\Omega_{\bLa_A,\bm z}$ in $\s\Gr(N,d)$ and $\s\Omega_{\bLa_{A,\bm k},\bm z}$ in $\s\Gr(N,d+|\bs k|)$.

	\medskip

If $(\bLa,\bm k)$ is $d$-nontrivial then $\bLa_{A,\bm k}$ is $d$-nontrivial. The converse is not true.

\begin{eg} 
For this example we write the highest weights in terms of fundamental weights, e.g. $(1,0,0,1)=\omega_1+\omega_4$. 
We also use $\sln_N$-modules instead of $\gl_N$-modules, since the spaces of invariants are the same.

For $N=4$ and $\g_4=\mathfrak{so}_5$ of type ${\rm B}_2$, we have
$$
\dim (V_{(2,0)}\otimes V_{(1,0)}\otimes V_{(2,0)})^{\g_4}=0 \ \  {\rm and}\ \ 
\dim (V_{(2,0,2)}\otimes V_{(1,0,1)}\otimes V_{(2,0,2)})^{\sln_4}=2.
$$Let $\bLa=((2,0),(1,0),(2,0))$. Then $\bLa_A$ is $9$-nontrivial, but $(\bLa,(0,0,0))$ is not.

Similarly, for $N=5$ and $\g_5=\mathfrak{sp}_4$ of type ${\rm C}_2$, we have
$$
\dim (V_{(1,0)}\otimes V_{(0,1)}\otimes V_{(0,1)})^{\g_5}=0 \ \  {\rm and}\ \ 
\dim (V_{(1,0,0,1)}\otimes V_{(0,1,1,0)}\otimes V_{(0,1,1,0)})^{\sln_5}=2.
$$
Let $\bLa=((1,0),(0,1),(1,0))$. Then $\bLa_A$  is $8$-nontrivial, but $(\bLa,(0,0,0))$ is not.
\end{eg}

Let $\mu^{(1)},\dots,\mu^{(a)}$ be all distinct partitions in $\bLa_{A,\bm k}$. Let $n_i$ be the number of occurrences of ${\mu}^{(i)}$ in $\bLa_{A,\bm k}$, then $\sum_{i=1}^an_i=n$. Denote $\bm n=(n_1,\dots,n_a)$, we shall write $\bLa_{A,\bm k}$ in the following order: $\la_{A,k_i}^{(i)}=\mu^{(j)}$ for $\sum_{s=1}^{j-1}n_s+1\lle i\lle \sum_{s=1}^{j}n_s$, $j=1,\dots,a$.
	
	\begin{prop}\label{prop bij BC}
		Suppose $(\bLa,\bm k)$ is $d$-nontrivial. The set $\s\Omega_{\bLa,\bm k}$ is a ramified covering of ${\mathring{\bP}}_n/S_{\bm n}$.   Moreover, the degree of the covering is equal to $\dim(V_\bLa)^{\g_N}$.  In particular, $\dim \s\Omega_{\bLa,\bm k}=n$. Over $\R{\mathring{\bP}}_n/S_{\bm n}$, this covering is unramified of the same degree, moreover all points in fibers are real. 
	\end{prop}
	\begin{proof}
		The proposition follows from Theorem \ref{bi rep sgr}, Lemma \ref{lem BC completeness}, and Theorem 1.1 of \cite{MTV2}.
	\end{proof}

		We find strata $\s\Omega_{\bLa,\bm k}\subset\sGr(N,d)$ of the largest dimension. 
	\begin{lem}\label{thm dim sgr}
		If $N=2r$, then the $d$-nontrivial stratum $\s\Omega_{\bs \La,\bm k}\subset\sGr(N,d)$ with the largest dimension has $(\la^{(s)},k_s)=(\omega_r,0)$, $s=1,\dots,2(d-N)$. In particular, the dimension of this stratum is $2(d-N)$.
		
		If $N=2r+1$, the $d$-nontrivial strata $\s\Omega_{\bs \La,\bm k}\subset\sGr(N,d)$ with the largest dimension have $(\la^{(s)},k_s)$ equal to either $(\omega_{j_s},0)$ with some $j_s\in\{1,\dots,r\}$, or to $(0,1)$, for 	$s=1,\dots,d-N$. Each such stratum is either empty or has dimension $d-N$. There is at least one nonempty stratum of this dimension, and if $d>N+1$ then more than one.
	\end{lem}
	\begin{proof}
		By Proposition \ref{prop bij BC}, we are going to find the maximal $n$ such that $(\bLa,\bm k)$ is $d$-nontrivial, where $\bLa=(\la^{(1)},\dots,\la^{(n)})$ is a sequence of dominant integral $\g_N$-weights and $\bm k=(k_1,\dots,k_n)$ is an $n$-tuple of nonnegative integers. Since $\bLa_{A,\bm k}$ is $d$-nontrivial, it follows that $\la^{(s)}\ne 0$ or $\la^{(s)}= 0$ and $k_s>0$, for all $s=1,\dots,n$.

		Suppose $N=2r$. If $\la^{(s)}\ne 0$, we have\[
		|\la_{A,k_s}^{(s)}|\gge |\la_{A,0}^{(s)}|=r\Big(2\sum_{i=1}^{r-1}\langle \la^{(s)},\calpha_i\rangle+\langle \la^{(s)},\calpha_r\rangle\Big)\gge r.
		\]If $k_s>0$, then $|\la_{A,k_s}^{(s)}|\gge 2rk_s\gge 2r$. Therefore, it follows that
		\[
		rn\lle \sum_{s=1}^n|\la_{A,k_s}^{(s)}|=|\bLa_{A,\bm k}|=(d-N)N.
		\]Hence $n\lle 2(d-N)$.
		
		If we set $\la^{(s)}=w_r$ and $k_s=0$ for all $s=1,\dots,2(d-N)$. Then $(\bs \La, \bs k)$ is $d$-nontrivial
	 since
		\[
		\dim(V_{\omega_r}\otimes V_{\omega_r})^{\mathfrak{so}_{2r+1}}=1.
		\]
		
		Now let us consider $N=2r+1$, $r\gge 1$.  Similarly, if $\la^{(s)}\ne 0$, we have\[
		|\la_{A,k_s}^{(s)}|\gge |\la_{A,0}^{(s)}|=(2r+1)\sum_{i=1}^r\langle \la^{(s)},\calpha_i\rangle\gge 2r+1.
		\]If $k_s>0$, then $|\la_{A,k_s}^{(s)}|\gge (2r+1)k_s\gge 2r+1$. It follows that
		\[
		(2r+1)n\lle \sum_{s=1}^n|\la_{A,k_s}^{(s)}|=|\bLa_{A,\bm k}|=(d-N)N.
		\]Hence $n\lle d-N$.
	Clearly, the equality is achieved only for the $(\bLa,\bs k)$ described in the statement of the lemma. Note that if $(\la^{(s)},k_s)=(0,1)$ for all $s=1,\dots,d-N$, then $(\bLa,\bs k)$ is $d$-nontrivial and therefore nonempty. If $d>N+1$ we also have $d$-nontrivial tuples parameterized by $i=1,\dots,r$, such that $(\la^{(s)},k_s)=(0,1)$, $s=3,\dots, d-N$, and $(\la^{(s)},k_s)=(\omega_i,0)$, $s=1,2$.
	\end{proof}

	\subsection{The $\g_N$-stratification of self-dual Grassmannian}
	The following theorem follows directly from Theorems \ref{thm gr dec} and \ref{bi rep sgr}. 
	
	\begin{thm}\label{thm sgr dec}
		We have
		\beq\sGr(N,d)=\bigsqcup_{d\text{-nontrivial }(\bLa,\bm k)}\s\Omega_{\bLa,\bm k}.\eeq
		\qed
	\end{thm}
	
	Next, for a $d$-nontrivial $(\bLa,\bm k)$, we call the closure of $\s\Omega_{\bLa,\bm k}$ inside $\sGr(N,d)$, a {\it $\g_N$-cycle}.
	The $\g_N$-cycles $\overline{\s\Omega}_{\bLa,\bm k}$ are algebraic sets in $\sGr(N,d)$ and therefore in $\Gr(N,d)$. We 
	describe $\g_N$-cycles as unions of $\g_N$-strata similar to \eqref{eq A strata}.
		
		Define a partial order $\gge$ on the set of pairs $\{(\bLa,\bm k)\}$ as follows.
	Let $\bLa=(\la^{(1)},\dots,\la^{(n)})$, $\bs\Xi=(\xi^{(1)},\dots,\xi^{(m)})$ be two sequences of dominant integral $\g_N$-weights. Let $\bm k=(k_1,\dots,k_n)$, $\bm l=(l_1,\dots,l_m)$ be two tuples of nonnegative integers. We say that $(\bLa,\bm k)\gge (\bm\Xi,\bm l)$ if and only if there exists a partition $\{I_1,\dots,I_m\}$ of $\{1,2,\dots,n\}$ such that
	\[
	{ \rm {Hom}}_{\g_N}(V_{\xi^{(i)}},\bigotimes_{j\in I_i} V_{\la^{(j)}})\ne 0,\qquad |\xi^{(i)}_{A,l_i}|=\sum_{j\in I_i}|\la_{A,k_j}^{(j)}|,
	\]for $i=1,\dots,m$.
	
	If $(\bLa,\bm k)\gge (\bm\Xi,\bm l)$ are $d$-nontrivial, we call $\s\Omega_{\bm\Xi,\bm l}$ a \emph{degeneration} of $\s\Omega_{\bLa,\bm k}$. If we suppose further that $m=n-1$, we call $\s\Omega_{\bm\Xi,\bm l}$ a \emph{simple degeneration} of $\s\Omega_{\bLa,\bm k}$. 
	
	\begin{thm}\label{thm simple deg BC} 
		If $\s\Omega_{\bm\Xi,\bm l}$ is a degeneration of $\s\Omega_{\bLa,\bm k}$, then $\s\Omega_{\bm\Xi,\bm l}$ is contained in the $\g_N$-cycle $\overline{\s\Omega}_{\bLa,\bm k}$.
	\end{thm}
	
	Theorem \ref{thm simple deg BC} is proved in Section \ref{sec proof BC}.
	
	\begin{thm}\label{thm BC strata}
		For $d$-nontrivial $(\bLa,\bm k)$, we have
		\beq
		\overline{\s\Omega}_{\bLa,\bm k}=\bigsqcup_{\substack{(\bm\Xi,\bm l)\lle (\bLa,\bm k),\\ d\text{-nontrivial }(\bm\Xi,\bm l)}}\s\Omega_{\bm\Xi,\bm l}.
		\eeq
	\end{thm}
Theorem \ref{thm BC strata}  is proved in Section \ref{sec proof BC}.

\medskip

Theorems \ref{thm sgr dec} and \ref{thm BC strata} imply that the subsets $\s\Omega_{\bLa,\bm k}$ with $d$-nontrivial $(\bLa,\bm k)$ give a stratification of $\sGr(N,d)$, similar to the $\gl_{N}$-stratification of $\Gr(N,d)$, see \eqref{eq decom gr} and \eqref{eq A strata}. We call it the {\it $\g_N$-stratification of $\s\Gr(N,d)$}.
	
\begin{eg}\label{eg gl-strata}
The following picture gives an example for $\mathfrak{so}_{5}$-stratification of $\sGr(4,6)$. In the following picture, we write $((\la^{(1)})_{k_1},\dots,(\la^{(n)})_{k_n})$ for $\s\Omega_{\bLa,\bm k}$. We also simply write $\la^{(s)}$ for $(\la^{(s)})_{0}$. For instance, $((0,1)_1,(0,1))$ represents $\s\Omega_{\bLa,\bm k}$ where $\bLa=((0,1),(0,1))$ and $\bm k=(1,0)$. The solid arrows represent simple degenerations. Unlike the picture in Example \ref{A ex} we do not include here the pairs of sequences which are not $6$-nontrivial, as there are too many of them. 
		
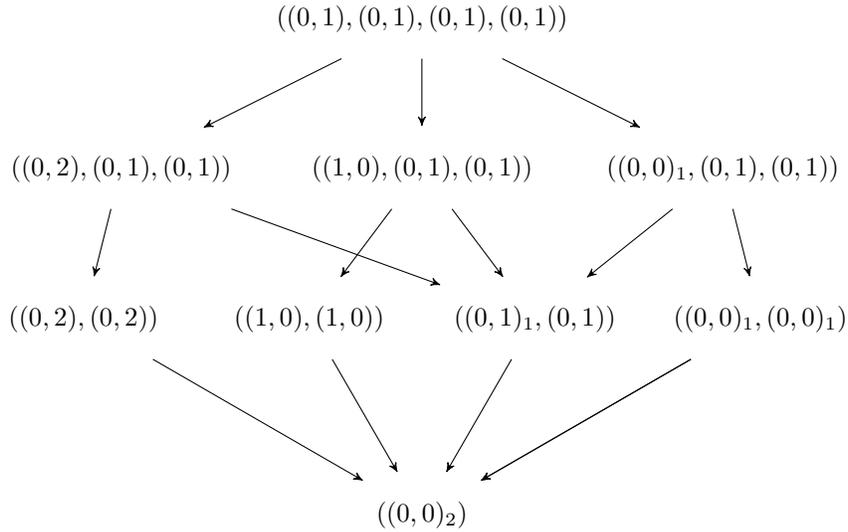
\begin{center}
\begin{tikzpicture}[->,>=stealth',shorten >=1pt,auto,node distance=2.8cm]
\tikzstyle{every state}=[rectangle,draw=none,text=black]
			
\node[state]         (S) at (-4, 4.6)              {$((0,1),(0,1),(0,1),(0,1))$};
\node[state]         (xin1) at (-8, 2.6)           {$((0,2),(0,1),(0,1))$};
\node[state]         (xin2) at (-4, 2.6)        {$((1,0),(0,1),(0,1))$};
\node[state]         (xin3) at (0, 2.6)        {$((0,0)_1,(0,1),(0,1))$};
\node[state]         (xout1) at (0.5,0.6 )          {$((0,0)_1,(0,0)_1)$};
\node[state]         (xout2) at (-5.5, 0.6)        {$((1,0),(1,0))$};
\node[state]         (xout6) at (-8.5, 0.6)        {$((0,2),(0,2))$};
\node[state]         (xout8) at (-2.5, 0.6)        {$((0,1)_1,(0,1))$};
\node[state]         (DC) at (-4, -2)           {$((0,0)_2)$};
			
\path
(S) edge                  node[xshift=-1.9cm,yshift=0.2cm]{} (xin1)
edge              node[xshift=0.4cm,yshift=-0.4cm]{} (xin2)
edge              node[xshift=0.4cm,yshift=-0.4cm]{} (xin3)
(xin1) edge  (xout6)
(xin2) edge  node[xshift=0.15cm,yshift=0.3cm]{} (xout2)
(xout1) edge   node[xshift=0cm,yshift=0cm]{} (DC)
(xin3) edge (xout8)
(xin3) edge (xout1)
(xin1) edge (xout8)
(xin2) edge (xout8)			
(xout6) edge (DC)
(xout1) edge (DC)
(xout2) edge (DC)
(xout8) edge (DC);		
\end{tikzpicture}
\end{center}
In particular, the stratum $\s\Omega_{((0,1),(0,1),(0,1),(0,1))}$ is dense in $\sGr(4,6)$.	
\end{eg}
	\begin{prop}\label{dense} If $N=2r$ is even, then the stratum
	$\s\Omega_{\bLa,\bm k}$ with $(\la^{(s)},k_s)=(\omega_r,0)$, where $s=1,\dots,2(d-N),$ is dense in $\sGr(N,d)$.
	\end{prop}
	\begin{proof}
	For $N=2r$, one has the $\g_N$-module decomposition
	\beq\label{eq dec last}
	V_{\omega_r}\otimes V_{\omega_r}=V_{2\omega_r}\oplus V_{\omega_1}\oplus \dots\oplus V_{\omega_{r-1}}\oplus V_{(0,\dots,0)}.
	\eeq
	It is clear that $(\bLa,\bm k)$ is $d$-nontrivial. It also follows from \eqref{eq dec last} that if $(\bm\Xi,\bm l)$ is $d$-nontrivial then $(\bLa,\bm k)\gge (\bm\Xi,\bm l)$. The proposition follows from Theorems \ref{thm sgr dec} and \ref{thm BC strata}.
\end{proof}

\begin{rem}\label{sd translations}
	The group of affine translations, see Remark \ref{translations}, preserves the self-dual Grassmannian $\s\Gr(N,d)$ and the strata $\s\Omega_{\bLa,\bs k}$.
	\end{rem}
\subsection{The $\g_N$-stratification of $\sGr(N,d)$ and the Wronski map}
Let $\bLa=(\la^{(1)},\dots,\la^{(n)})$ be a sequence of dominant integral $\g_N$-weights and let $\bm k=(k_1,\dots,k_n)$ be an $n$-tuple of nonnegative integers. Let $\bs z=(z_1,\dots,z_n)\in \mathring{\bP}_n$.

 Recall that $\la^{(s)}_i=\langle \la^{(s)},\calpha_i\rangle$. If $X\in\s\Omega_{\bLa,\bm k,\bm z}$, one has
 \[
 \Wr(X)=\begin{cases}

\Big( \prod\limits_{s=1}^n(x-z_s)^{\la^{(s)}_1+\dots+\la^{(s)}_{r}+k_s}\Big)^{N},\quad &\text{if }N=2r+1;\\
\Big( \prod\limits_{s=1}^n(x-z_s)^{2\la^{(s)}_1+\dots+2\la^{(s)}_{r-1}+\la^{(s)}_{r}+2k_s}\Big)^r,\quad & \text{if }N=2r.
 \end{cases}
 \]
 
 We define the \emph{reduced Wronski map} $\overline{\Wr}$ as follows. 
 
 If $N=2r+1$, the reduced Wronski map
 \[
 \overline{\Wr}:\sGr(N,d)\to \Gr(1,d-N+1)
 \]is sending $X\in \sGr(N,d)$ to $\C(\Wr(X))^{1/N}$.

 If $N=2r$, the reduced Wronski map
 \[
 \overline{\Wr}:\sGr(N,d)\to \Gr(1,2(d-N)+1)
 \]is sending $X\in \sGr(N,d)$ to $\C(\Wr(X))^{1/r}$.
 
 The reduced Wronski map is also a finite map. 

  For $N=2r$, the degree of the reduced Wronski map is given by $\dim (V_{\omega_r}^{\otimes{2(d-N)}})^{\g_{N}}$. This dimension is given by, see \cite{KLP},
  \beq\label{BC degree}
  (N-1)!!\prod_{1\lle i<j\lle r} \big((j-i)(N-i-j+1)\big)\prod_{k=0}^{r-1}\ \frac{(2(d-N+k))!}{(d-k-1)!(d-N+k)!}.
  \eeq
 
 Let $\tl d=d-N+1$ if $N=2r+1$ and $\tl d=2(d-N)+1$ if $N=2r$. Let $\bm m=(m_1,\dots,m_n)$ be an unordered sequence of positive integers such that $|\bm m|=\tl d-1$.
 
 Similar to Section \ref{sec wronski gr}, we have the following proposition.
 
	\begin{prop}\label{B preimage}
	The preimage of the stratum $\Omega_{\bm m}$ of $\Gr(1,\tl d)$ under the reduced Wronski map is a union of all strata $\s\Omega_{\bLa,\bm k}$ of $\sGr(N,d)$ such that $|\la^{(s)}_{A,k_s}|=N m_s$, $s=1,\dots,n$, if $N$ is odd and such that $|\la^{(s)}_{A,k_s}|=r m_s$, $s=1,\dots,n$, if $N=2r$ is even. \qed
	\end{prop}
	
Let  $\bLa=(\la^{(1)},\dots,\la^{(n)})$ be an unordered sequence of dominant integral $\g_N$-weights and $\bm k=(k_1,\dots,k_n)$ a sequence of nonnegative integers.
Let $a$ be the number of distinct pairs in the set $\{(\la^{(s)},k_s),\ s=1,\dots,n\}$. We can assume that $(\la^{(1)},k_1),\dots,(\la^{(a)},k_a)$ are all distinct, and let $n_1,\dots,n_a$ be their multiplicities, $n_1+\dots+n_a=n$. 

Consider the unordered set of integers $\bs m=(m_1,\dots,m_n)$,
where $Nm_s=|\la^{(s)}_{A,k_s}|$ if $N$ is odd or $rm_s=|\la^{(s)}_{A,k_s}|$ if $N=2r$ is even. Consider the stratum $\Omega_{\bs m}$ in $\Gr(1,\tl d)$, corresponding to polynomials with $n$ distinct roots of multiplicities $m_1,\dots,m_n$. 

\begin{prop}
\label{prop cov}
Let $(\bLa,\bs k)$ be $d$-nontrivial. Then the reduced Wronski map 
$\overline{\Wr}|_{\s\Omega_{\bLa,\bs k}} : \s\Omega_{\bLa,\bs k}\to \Omega_{\bs m}$ is a ramified covering
 of degree $b(\bLa_{A,\bs k}) \dim (V_{\bLa})^{\g_N}$, where $b(\bLa_{A,\bs k})$ is given by \eqref{sym co}.  
	\end{prop}

	\begin{proof}
		The statement follows from Theorem \ref{bi rep sgr}, Lemma \ref{lem BC completeness}, and Proposition \ref{B preimage}.
	\end{proof}
	
In other words, the $\mathfrak{g}_N$-stratification of $\s\Gr(N,d)$ given by Theorems \ref{thm sgr dec} and \ref{thm BC strata},
is adjacent to the swallowtail $\gl_1$-stratification of $\Gr(1,\tl d)$ and the reduced Wronski map. 

\subsection{Self-dual Grassmannian for $N=3$}\label{sec proof N=3} Let $N=3$ and $\g_3=\sln_2$. We identify the dominant integral $\sln_2$-weights with nonnegative integers. Let $\bLa=(\la^{(1)},\dots,\la^{(n)},\la)$ be a sequence of nonnegative integers and $\bm z=(z_1,\dots,z_n,\infty)\in {\mathring{\bP}}_{n+1}$. 
	
	Choose $d$ large enough so that $k:=d-3-\sum_{s=1}^n\la^{(s)}-\la\gge 0$. Let $\bs k=(0,\dots,0,k)$.
	Then $\bLa_{A,\bs k}$ has coordinates
	\[
	\la_{A}^{(s)}=(2\la^{(s)},\la^{(s)},0),\quad s=1,\dots,n,
	\]
	\[
	\la_{A, k}=\Big(d-3-\sum_{s=1}^n\la^{(s)}+\la,d-3-\sum_{s=1}^n\la^{(s)},d-3-\sum_{s=1}^n\la^{(s)}-\la\Big).
	\]
	
	Note that we always have $|\bLa_{A,\bs k}|=3(d-3)$ and spaces of polynomials in $\s\Omega_{\bLa,\bs k,\bm z}$ are pure self-dual spaces.
	\begin{thm}\label{thm psGr N=3}
		There exists a bijection between the common eigenvectors in $(V_{\bLa,\bm z})^{\sln_2}$ of the $\gl_2$ Bethe algebra $\cB$ and $\s\Omega_{\bLa,\bs k,\bs z}$.
	\end{thm}
	\begin{proof}
		Let $X\in\s\Omega_{\bLa,\bs k,\bm z}$, and let $\bm T=(T_1(x),T_2(x),T_3(x))$ be associated with $\bLa_{A,\bm k},\bm z$, then\[T_1(x)=T_2(x)=\prod_{s=1}^n(x-z_s)^{\la^{(s)}}.\]
		
		Following Section 6 of \cite{MV}, let $\bm u=(u_1,u_2,u_3)$ be a Witt basis of $X$, one has
		\[
		\Wr(u_1,u_2)=T_1u_1,\quad \Wr(u_1,u_3)=T_1u_2,\quad\Wr(u_2,u_3)=T_1u_3.
		\]Let $y(x,c)=u_1+cu_2+\frac{c^2}{2}u_3$, it follows from Lemma 6.15 of \cite{MV} that
		\[\Wr\Big(y(x,c),\frac{\pa y}{\pa c}(x,c)\Big)=T_1y(x,c).\]Since $X$ has no base points, there must exist $c'\in\C$ such that $y(x,c')$ and $T_1(x)$ do not have common roots. It follows from Lemma 6.16 of \cite{MV} that $y(x,c')=p^2$ and $y(x,c)=(p+(c-c')q)^2$ for suitable polynomials $p(x),q(x)$ satisfying $\Wr(p,q)=2T_1$. In particular, $\{p^2,pq,q^2\}$ is a basis of $X$. Without loss of generality, we can assume that $\deg p<\deg q$. Then
		\[
		\deg p=\frac{1}{2}\Big(\sum_{s=1}^n\la^{(s)}-\la\Big),\quad \deg q=\frac{1}{2}\Big(\sum_{s=1}^n\la^{(s)}+\la\Big)+1.
		\]Since $X$ has no base points, $p$ and $q$ do not have common roots. Combining with the equality $\Wr(p,q)=2T_1$, one has that the space spanned by $p$ and $q$ has singular points at $z_1,\dots,z_n$ and $\infty$ only. Moreover, the exponents at $z_s$, $s=1,\dots,n$, are equal to $0,\la^{(s)}+1$, and the exponents at $\infty$ are equal to $-\deg p,-\deg q$.
		
		By Theorem \ref{thm bijection}, the space $\mathrm{span}\{p,q\}$ corresponds to a common eigenvector of the $\gl_2$ Bethe subalgebra in the subspace $\big(\bigotimes_{s=1}^n V_{(\la^{(s)},0)}(z_s)\otimes V_{(d-2-\deg p,d-1-\deg q)}(\infty)\big)^{\sln_2}$.
		
		Conversely, given a common eigenvector of the $\gl_2$ Bethe algebra in $(V_{\bLa,\bm z})^{\sln_2}$, by Theorem \ref{thm bijection}, it corresponds to a space $\tl X$ of polynomials in $\Gr(2,d)$ without base points. Let $\{p,q\}$ be a basis of $\tl X$, define a space of polynomials $\mathrm{span}\{p^2,pq,q^2\}$ in $\Gr(3,d)$. It is easy to see that $\mathrm{span}\{p^2,pq,q^2\}\in \s\Omega_{\bLa,\bs k,\bm z}$ is a pure self-dual space.
	\end{proof}
	
	Let $X\in \Gr(2,d)$, denote by $X^2$ the space spanned by $f^2$ for all polynomials $f\in X$. It is clear that $X^2\in\sGr(3,2d-1)$. Define 
	\beq\label{map N=3}
	\pi:\Gr(2,d)\to \sGr(3,2d-1)
	\eeq
	by sending $X$ to $X^2$. The map $\pi$ is an injective algebraic map.
	\begin{cor}
	The map $\pi$ defines a bijection between the subset of spaces of polynomials without base points in $\Gr(2,d)$ and the subset of pure self-dual spaces in $\sGr(3,2d-1)$.\qed
	\end{cor}
	 Note that not all self-dual spaces in $\sGr(3,2d-1)$ can be expressed as $X^2$ for some $X\in\Gr(2,d)$ since the greatest common divisor of a self-dual space does not have to be a square of a polynomial.
	 
\subsection{Geometry and topology}\label{sec ex} It would be very interesting to determine the topology and geometry of the strata and cycles of $\Gr(N,d)$ and of $\s\Gr(N,d)$. In particular, it would be interesting to understand the geometry and topology of the self-dual Grassmannian $\s\Gr(N,d)$. Here are some simple examples of small dimension.

\medskip

Of course, $\sGr(N,N)=\Gr(N,N)$ is just one point. Also, $\s\Gr(2r+1,2r+2)$ is just ${\Bbb P}^1$.

\medskip

Consider $\s\Gr(2r,2r+1)$, $r\gge 1$. It has only two strata: $\s\Omega_{(\omega_r,\omega_r),(0,0)}$ and $\s\Omega_{(0),(1)}$. Moreover, the reduced Wronski map has degree 1 and defines a bijection: $\overline{\Wr} :\s\Gr(2r,2r+1)\to \Gr(1,3)$. In particular, the ${\mathfrak{so}}_{2r+1}$-stratification in this case is identified with the swallowtail $\gl_1$-stratification of quadratics. There are two strata: polynomials with two distinct roots and polynomials with one double root. 
Therefore through the reduced Wronski map, the self-dual Grassmannian $\s\Gr(2r,2r+1)$ can be identified with $ {\Bbb P}^2$ with coordinates $(a_0:a_1:a_2)$ and the stratum $\s\Omega_{(0),(1)}$ is a nonsingular curve of degree 2 given by the equation $a_1^2-4a_0a_2=0$. 

\medskip

Consider $\s\Gr(2r+1,2r+3)$, $r\gge 1$. In this case we have $r+2$ strata: $\s\Omega_{(\omega_i,\omega_i),(0,0)}$, $i=1,\dots,r$, $\s\Omega_{(0,0),(1,1)}$, and 
$\s\Omega_{(0),(2)}$. The reduced Wronski map  $\overline{\Wr} :\s\Gr(2r+1,2r+3)\to \Gr(1,3)$ restricted to any strata again 
has degree 1. Therefore, through the reduced Wronski map, the self-dual Grassmannian $\s\Gr(2r+1,2r+3)$ can be identified with $r+1$ copies of ${\Bbb P}^2$ all intersecting in the same nonsingular degree 2 curve corresponding to the stratum $\s\Omega_{(0),(2)}$. In particular, every 2-dimensional ${\mathfrak{sp}}_{2r}$-cycle is just ${\Bbb P}^2$.

\medskip

Consider $\s\Gr(2r+1,2r+4)$, $r\gge 1$. We have $\dim \s\Gr(2r+1,2r+4) =3$. 
This is the last case when for all strata the coverings of Proposition \ref{prop bij BC} have degree one. There are already many strata. For example,  consider $\sGr(5,8)$, that is $r=2$. There are four strata of dimension 3 
	corresponding to the following sequences of ${\mathfrak {sp}}_4$-weights and 3-tuples of nonnegative integers:
	\[
	\bLa_1=(\omega_1,\omega_1,0),\quad \bs k_1=(0,0,1);\qquad
	\bLa_2=(\omega_1,\omega_1,\omega_2),\quad \bm k_2=(0,0,0);
	\]	\[
	\bLa_3=(\omega_2,\omega_2,0),\quad \bm k_3=(0,0,1);\qquad
	\bLa_4=(0,0,0),\quad \bm k_4=(1,1,1).
	\]
By the reduced Wronski map, the stratum $\Omega_{\bLa_4,\bs k_4}$ is identified with the subset of 
$\Gr(1,4)$ represented by cubic polynomials without multiple roots and the cycle 
	$\overline{\Omega}_{\bLa_4,\bs k_4}$ with $\Gr(1,4)={\Bbb P}^3$. The stratification of $\overline{\Omega}_{\bLa_4,\bs k_4}$ is just the swallowtail of cubic polynomials.
However, for other three strata the reduced Wronski map has degree 3. Using instead the map in Proposition \ref{prop bij BC}, we identify each of these strata with ${\mathring{\bP}}_{3}/(\Z/2\Z)$ or with the subset of $\Gr(1,3)\times \Gr(1,2)$ represented by a pair of polynomials $(p_1,p_2)$, such that $\deg(p_1)\lle 2$, $\deg(p_2)\lle 1$ and such that 
all three roots (including infinity) of $p_1p_2$ are distinct. Then the corresponding ${\mathfrak {sp}}_4$-cycles $\overline{\Omega}_{\bLa_i,\bs k_i}$ , $i=1,2,3$, are identified with $\Gr(1,3)\times \Gr(1,2)={\Bbb P}^2\times {\Bbb P}^1$.
	
	\medskip
	
	A similar picture is observed for 3-dimensional strata in the case of $\s\Gr(2r,2r+2)$. Consider,	for example, $\Gr(2,4)$, see Example \ref{A ex}. Then the 4-dimensional stratum $\Omega_{(1,0),(1,0),(1,0),(1,0)}$ is dense and (relatively) complicated, as the corresponding covering in Proposition \ref{prop deg of cover A} has degree 2. But for the 3-dimensional strata the degrees are 1. Therefore,  $\Omega_{(2,0),(1,0),(1,0)}$ and  $\Omega_{(1,1),(1,0),(1,0)}$ are identified with ${\mathring{\bP}}_{3}/(\Z/2\Z)$ and the corresponding cycles are just $\Gr(1,3)\times \Gr(1,2)={\Bbb P}^2\times {\Bbb P}^1$.

	\section{More notation}\label{sec more notation}
	\subsection{Lie algebras}
	Let $\g$ and $\h$ be as in Section \ref{sec sla}. One has the Cartan decomposition $\g=\n_-\oplus\h\oplus\n_+$. Introduce also the positive and negative Borel subalgebras $\mathfrak b=\h\oplus \n_+$ and $\mathfrak b_-=\h\oplus \n_-$.
	
	Let $\mathscr G$ be a simple Lie group, $\mathscr B$ a Borel subgroup, and $ \mathscr N=[\mathscr B,\mathscr B]$ its unipotent radical, with the corresponding Lie algebras $\n_+\subset\mathfrak b\subset\g$. Let $\mathscr G$ act on $\g$ by adjoint action.
	
	Let $E_1,\dots,E_r\in \n_+$, $\calpha_1,\dots,\calpha_r\in \h$, $F_1,\dots,F_r\in\n_-$ be the Chevalley generators of $\g$. Let $p_{-1}$ be the regular nilpotent element $\sum_{i=1}^rF_i$. The set $p_{-1}+\mathfrak b=\{p_{-1}+b~|~b\in\mathfrak b\}$ is invariant under conjugation by elements of $\mathscr N$. Consider the quotient space $(p_{-1}+\mathfrak b)/\mathscr N$ and denote the $\mathscr N$-conjugacy class of $g\in p_{-1}+\mathfrak b$ by $[g]_\g$.
	
	Let $\check{\mathcal P}=\{\cla\in\mathfrak h|\langle \alpha_i,\cla\rangle\in\mathbb Z,\ i=1,\dots,r\}$ and $\check{\mathcal P}^+=\{\cla\in\mathfrak h|\langle \alpha_i,\cla\rangle\in\mathbb Z_{\gge 0},\ i=1,\dots,r\}$ be the coweight lattice and the cone of dominant integral coweights. Let $\rho\in\h^*$ and $\crho\in\h$ be the Weyl vector and covector such that $\langle \rho,\calpha_i\rangle =1$ and $\langle \alpha_i,\crho\rangle =1$, $i=1,\dots,r$.
	
	The Weyl group $\mc W\subset\mathrm{Aut}(\mathfrak h^*)$ is generated by simple reflections $s_i$, $i=1,\dots,r$,
	\[s_i(\la)=\la-\langle \la,\calpha_i\rangle \alpha_i,\quad \la\in\mathfrak h^*.\]The restriction of the bilinear form $(\cdot,\cdot)$ to $\h$ is nondegenerate and induces an isomorphism $\h\cong\h^*$. The action of $\mathcal W$ on $\h$ is given by $s_i(\cmu)=\cmu-\langle \alpha_i,\cmu\rangle \calpha_i$ for $\cmu\in\h$. We use the notation
	\[w\cdot\la=w(\la+\rho)-\rho,\quad w\cdot\cla=w(\cla+\crho)-\crho,\quad w\in\mc W,~\la\in\h^*,~\cla\in\h,\]
	for the shifted action of the Weyl group on $\h^*$ and $\h$, respectively.
	
	Let $^t\g=\g(^tA)$ be the Langlands dual Lie algebra of $\g$, then $^t(\mathfrak{so}_{2r+1})=\mathfrak{sp}_{2r}$ and $^t(\mathfrak{sp}_{2r})=\mathfrak{so}_{2r+1}$. A system of simple roots of $^t\g$ is $\calpha_1,\dots,\calpha_r$ with the corresponding coroots $\alpha_1,\dots,\alpha_r$. A coweight $\cla\in\h$ of $\g$ can be identified with a weight of $^t\g$.
	
	For a vector space $X$ we denote by $\cM(X)$ the space of $X$-valued meromorphic functions on $\mathbb P^1$. For a group $R$ we denote by $R(\cM)$ the group of $R$-valued meromorphic functions on $\bP^1$.
	
	\subsection{$\mathfrak{sp}_{2r}$ as a subalgebra of $\sln_{2r}$}\label{sec id C}Let $v_1,\dots,v_{2r}$ be a basis of $\C^{2r}$. Define a nondegenerate skew-symmetric form $\chi$ on $\C^{2r}$ by
	\[\chi(v_i,v_j)=(-1)^{i+1}\delta_{i,2r+1-j},\quad i,j=1,\dots,2r.\]The special symplectic Lie algebra $\g=\mathfrak{sp}_{2r}$ by definition consists of all endomorphisms $K$ of $\C^{2r}$ such that $\chi(Kv,v')+\chi(v,Kv')=0$ for all $v,v'\in\C^{2r}$. This identifies $\mathfrak{sp}_{2r}$ with a Lie subalgebra of $\sln_{2r}$.
	
	Denote $E_{ij}$ the matrix with zero entries except 1 at the intersection of the $i$-th row and $j$-th column.
	
	The Chevalley generators of $\g=\mathfrak{sp}_{2r}$ are given by
	\[
	E_i=E_{i,i+1}+E_{2r-i,2r+1-i},\quad  F_i=E_{i+1,i}+E_{2r+1-i,2r-i},\quad i=1,\dots,r-1,\]
	\[E_r=E_{r,r+1},\qquad F_r=E_{r+1,r},\]
	\[ \calpha_j=E_{jj}-E_{j+1,j+1}+E_{2r-j,2r-j}-E_{2r+1-j,2r+1-j},~~\calpha_r=E_{rr}-E_{r+1,r+1},~~j=1,\dots,r-1.
	\]Moreover, a coweight $\cla\in\h$ can be written as
	\beq\label{eq id B}
	\cla=\sum_{i=1}^r\Big(\langle \alpha_i,\cla\rangle+\dots+\langle \alpha_{r-1},\cla\rangle+\langle \alpha_r,\cla\rangle/2\Big)(E_{ii}-E_{2r+1-i,2r+1-i}).
	\eeq
	In particular,
	\[
	\crho=\sum_{i=1}^r\frac{2r-2i+1}{2}(E_{ii}-E_{2r+1-i,2r+1-i}).
	\]
	For convenience, we denote the coefficient of $E_{ii}$ in the right hand side of \eqref{eq id B} by $(\cla)_{ii}$, for $i=1,\dots,2r$.
	
	\subsection{$\mathfrak{so}_{2r+1}$ as a subalgebra of $\sln_{2r+1}$}\label{sec id B}Let $v_1,\dots,v_{2r+1}$ be a basis of $\C^{2r+1}$. Define a nondegenerate symmetric form $\chi$ on $\C^{2r+1}$ by
	\[\chi(v_i,v_j)=(-1)^{i+1}\delta_{i,2r+2-j},\quad i,j=1,\dots,2r+1.\]The special orthogonal Lie algebra $\g=\mathfrak{so}_{2r+1}$ by definition consists of all endomorphisms $K$ of $\C^{2r+1}$ such that $\chi(Kv,v')+\chi(v,Kv')=0$ for all $v,v'\in\C^{2r+1}$. This identifies $\mathfrak{so}_{2r+1}$ with a Lie subalgebra of $\sln_{2r+1}$.
	
	Denote $E_{ij}$ the matrix with zero entries except 1 at the intersection of the $i$-th row and $j$-th column.
	
	The Chevalley generators of $\g=\mathfrak{so}_{2r+1}$ are given by
	\[
	E_i=E_{i,i+1}+E_{2r+1-i,2r+2-i},\quad F_i=E_{i+1,i}+E_{2r+2-i,2r+1-i},\quad i=1,\dots,r,\]
	\[ \calpha_j=E_{jj}-E_{j+1,j+1}+E_{2r+1-j,2r+1-j}-E_{2r+2-j,2r+2-j},\quad j=1,\dots,r.
	\]Moreover, a coweight $\cla\in\h$ can be written as
	\beq\label{eq id C}
	\cla=\sum_{i=1}^r\Big(\langle \alpha_i,\cla\rangle+\dots+\langle \alpha_r,\cla\rangle\Big)(E_{ii}-E_{2r+2-i,2r+2-i}).
	\eeq
	In particular,
	\[
	\crho=\sum_{i=1}^r(r+1-i)(E_{ii}-E_{2r+2-i,2r+2-i}).
	\]
	For convenience, we denote the coefficient of $E_{ii}$ in the right hand side of \eqref{eq id C} by $(\cla)_{ii}$, for $i=1,\dots,2r+1$.
	
	\subsection{Lemmas on spaces of polynomials}Let $\bLa=(\la^{(1)},\dots,\la^{(n)},\la)$ be a sequence of partitions with at most $N$ parts such that $|\bLa|=N(d-N)$ and let $\bm z=(z_1,\dots,z_n,\infty)\in {\mathring{\bP}}_{n+1}$.
	
	Given an $N$-dimensional space of polynomials $X$, denote by $\mc D_X$ the monic scalar differential operator of order $N$ with kernel $X$. The operator $\mc D_X$ is a monodromy-free Fuchsian differential
	operator with rational coefficients.
	
	\begin{lem}\label{lem exp}
		A subspace $X\subset\C_d[x]$ is a point of $\Omega_{\bLa,\bs z}$ if and only if the operator $\mc D_X$ is Fuchsian,  regular in $\C\setminus\{z_1,\dots,z_n\}$, the exponents at $z_s$, $s=1,\dots,n$, being equal to $\la_N^{(s)},\la_{N-1}^{(s)}+1,\dots,\la_1^{(s)}+N-1$, and the exponents at $\infty$ being equal to $1+\la_N-d,2+\la_{N-1}-d,\dots,N+\la_1-d$.\qed
	\end{lem}
	
	Let $\bs T=(T_1,\dots,T_N)$  be associated with $\bLa,\bm z$, see \eqref{eq def T}. Let $\Gamma=\{u_1,\dots,u_N\}$ be a basis of $X\in \Omega_{\bLa,\bm z}$, define a sequence of polynomials
	\beq\label{eq y}
	y_{N-i}=\Wr^\dag(u_1,\dots,u_i),\quad i=1,\dots,N-1.
	\eeq
	Denote $(y_1,\dots,y_{N-1})$ by $\bs y_\Gamma$. We say that $\bs y_\Gamma$ is {\it constructed from the basis} $\Gamma$.
	\begin{lem}[\cite{MV}]\label{lem D_X}
		Suppose $X\in \Omega_{\bLa,\bm z}$ and let $\Gamma=\{u_1,\dots,u_N\}$ be a basis of $X$. If $\bs y_\Gamma=(y_1,\dots,y_{N-1})$ is constructed from $\Gamma$, then
		\begin{align*}
		\mc D_X=\Big(\pa_x-\ln'\Big(\frac{T_1\cdots T_N}{y_1}\Big)\Big)&\Big(\pa_x-\ln'\Big(\frac{y_1T_2\cdots T_{N}}{y_2}\Big)\Big)\times\dots\\
		&\times\Big(\pa_x-\ln'\Big(\frac{y_{N-2}T_{N-1}T_{N}}{y_{N-1}}\Big)\Big)\Big(\pa_x-\ln'(y_{N-1}T_{N})\Big).
		\end{align*}\qed
	\end{lem}
	
	Let $\mc D=\pa_x^N+\sum_{i=1}^Nh_i(x)\pa_x^{N-i}$ be a differential operator with meromorphic coefficients. The operator $\mc D^*=\pa_x^N+\sum_{i=1}^N(-1)^i\pa_x^{N-i}h_i(x)$ is called the {\it formal conjugate to} $\mc D$.
	
	\begin{lem}\label{lem dual conj}
		Let $X\in \Omega_{\bLa,\bm z}$ and let $\{u_1,\dots,u_N\}$ be a basis of $X$, then
		\[
		\frac{\Wr(u_1,\dots,\widehat{u_i},\dots,u_N)}{\Wr(u_1,\dots,u_N)},\quad i=1,\dots,N,
		\]form a basis of $\Ker ((\mc D_{X})^*)$. The symbol $\widehat{u_i}$ means that $u_i$ is skipped. Moreover, given an arbitrary factorization of $\mc D_X$ to linear factors,
		$\mc D_X=(\pa_x+f_1)(\pa_x+f_2)\dots(\pa_x+f_N)$, we have $(\mc D_{X})^*= (\pa_x-f_N)(\pa_x-f_{N-1})\dots(\pa_x-f_1)$.
	\end{lem}
	\begin{proof}
		The first statement follows from Theorem 3.14 of \cite{MTV4}. The second statement follows from the first statement and Lemma A.5 of \cite{MV}.
	\end{proof}
	\begin{lem}\label{lem dual-oper}
		Let $X\in \Omega_{\bLa,\bm z}$. Then $$\mc D_{X^\dag}=(T_1\cdots T_N)\cdot \big(\mc D_{X}\big)^*\cdot (T_1\cdots T_N)^{-1}.$$
	\end{lem}
	\begin{proof}
The statement follows from Lemma \ref{lem dual conj} and the definition of $X^\dag$.
	\end{proof}
	\begin{lem}\label{lem y}
		Suppose $X\in \Omega_{\bLa,\bm z}$ is a pure self-dual space and $z$ is an arbitrary complex number, then there exists a basis $\Gamma=\{u_1,\dots,u_N\}$ of $X$ such that for $\bm y_\Gamma=(y_1,\dots,y_{N-1})$ given by \eqref{eq y}, we have $y_i=y_{N-i}$ and $y_i(z)\ne 0$ for every $i=1,\dots,N-1$.
	\end{lem}
	\begin{proof}
		The lemma follows from the proofs of Theorem 8.2 and Theorem 8.3 of \cite{MV}.
	\end{proof}

	\section{$\g$-oper}\label{sec oper}
	We fix $N$, $N\gge 4$, and set $\g$ to be the Langlands dual of $\g_N$. Explicitly, $\g=\mathfrak{sp}_{2r}$ if $N=2r$ and $\g=\mathfrak{so}_{2r+1}$ if $N=2r+1$.
	\subsection{Miura $\g$-oper}\label{sec Miura}
	Fix a global coordinate $x$ on $\C\subset\bP^1$. Consider the following subset of differential operators
	\[
	\text{op}_{\g}(\bP^1)=\{\pa_x+p_{-1}+\bs v~|~ \bs v\in\cM(\mathfrak b)\}.
	\]This set is stable under the gauge action of the unipotent subgroup $\mathscr N(\cM)\subset \mathscr G(\cM)$. The space of $\g$-{\it opers} is defined as the quotient space $\Opg(\bP^1):=\opg(\bP^1)/\mathscr N(\cM)$. We denote by $[\nabla]$ the class of $\nabla\in \opg(\bP^1)$ in $\Opg(\bP^1)$.
	
	We say that $\nabla=\pa_x+p_{-1}+\bs v\in\opg(\bP^1)$ is {\it regular} at $z\in \bP^1$ if $\bs v$ has no pole at $z$. A $\g$-oper $[\nabla]$ is said to be {\it regular} at $z$ if there exists $f\in \mathscr N(\cM)$ such that $f^{-1}\cdot\nabla\cdot f$ is regular at $z$.
	
	Let $\nabla=\pa_x+p_{-1}+\bs v$ be a representative of a $\g$-oper $[\nabla]$. Consider $\nabla$ as a $\mathscr G$-connection on the trivial principal bundle $p:\mathscr G\times \bP^1\to \bP^1$. The connection has singularities at the set $\mathrm{Sing}\subset \C$ where the function $\bs v$ has poles (and maybe at infinity). Parallel translations with respect to the connection define the monodromy representation $\pi_1(\C\setminus\mathrm{Sing})\to \mathscr G$. Its image is called the {\it monodromy group} of $\nabla$. If the monodromy group of one of the representatives of $[\nabla]$ is contained in the center of $\mathscr G$, we say that $[\nabla]$ is a {\it monodromy-free} $\g$-oper.
	
	A {\it Miura $\g$-oper} is a differential operator of the form $\nabla=\pa_x+p_{-1}+\bs v$, where $\bs v\in\cM(\h)$. 
	
	A $\g$-oper $[\nabla]$ has {\it regular singularity} at $z\in\bP^1\setminus\{\infty\}$, if there exists a representative $\nabla$ of $[\nabla]$ such that
	\[
	(x-z)^{\crho}\cdot\nabla\cdot (x-z)^{-\crho}=\pa_x+\frac{p_{-1}+\bs w}{x-z},
	\]where $\bs w\in\cM(\mathfrak b)$ is regular at $z$. The residue of $[\nabla]$ at $z$ is $[p_{-1}+\bs w(z)]_\g$. We denote the residue of $[\nabla]$ at $z$ by $\mathrm{res}_{z}[\nabla]$.
	
	Similarly, a $\g$-oper $[\nabla]$ has {\it regular singularity} at $\infty\in\bP^1$, if there exists a representative $\nabla$ of $[\nabla]$ such that
	\[
	x^{\crho}\cdot\nabla\cdot x^{-\crho}=\pa_x+\frac{p_{-1}+\bs {\tilde w} }{x},
	\]where $\bs{\tilde w}\in\cM(\mathfrak b)$ is regular at $\infty$. The residue of $[\nabla]$ at $\infty$ is $-[p_{-1}+\bs {\tilde w}(\infty)]_\g$. We denote the residue of $[\nabla]$ at $\infty$ by $\mathrm{res}_{\infty}[\nabla]$.
	
	\begin{lem}\label{lem res id}
		For any $\cla,\cmu\in\h$, we have $[p_{-1}-\crho-\cla]_\g=[p_{-1}-\crho-\cmu]_\g$ if and only if there exists $w\in\mc W$ such that $\cla=w\cdot \cmu$. \qed
	\end{lem}
	
	Hence we can write $[\cla]_{\mc W}$ for $[p_{-1}-\crho-\cla]_\g$. In particular, if $[\nabla]$ is regular at $z$, then $\mathrm{res}_z[\nabla]=[0]_{\mc W}$.
	
	Let $\check\bLa=(\cla^{(1)},\dots,\cla^{(n)},\cla)$ be a sequence of $n+1$ dominant integral $\g$-coweights and let $\bs z=(z_1,\dots,z_n,\infty)\in {\mathring{\bP}_{n+1}}$. Let $\text{Op}_\g(\mathbb{P}^1)_{\check\bLa,\bm{z}}^{\text{RS}}$ denote the set of all $\g$-opers with at most regular singularities at points $z_s$ and $\infty$ whose residues are given by
	\[
	\text{res}_{z_s}[\nabla]=[\cla^{(s)}]_{\mathcal W},\quad \text{res}_{\infty}[\nabla]=-[\cla]_{\mathcal W},\quad s=1,\dots,n,
	\]and which are regular elsewhere. Let $\text{Op}_\g(\mathbb{P}^1)_{\check\bLa,\bm{z}}\subset \text{Op}_\g(\mathbb{P}^1)_{\check\bLa,\bm{z}}^{\text{RS}}$ denote the subset consisting of those $\g$-opers which are also monodromy-free.
	
	\begin{lem}[\cite{F}]\label{lem sur}
		For every $\g$-oper $[\nabla]\in\Opg(\bP^1)_{\check\bLa,\bm{z}}$, there exists a Miura $\g$-oper as one of its representatives.\qed
	\end{lem}
	
	\begin{lem}[\cite{F}]\label{lem m-form}
		Let $\nabla$ be a Miura $\g$-oper, then $[\nabla]\in\Opg(\bP^1)_{\check\bLa,\bm{z}}^{\mathrm{RS}}$ if and only if the following conditions hold:
		\begin{enumerate}
			\item $\nabla$ is of the form
			\beq\label{eq m-form}
			\nabla=\pa_x+p_{-1}-\sum_{s=1}^n\frac{w_s\cdot\cla^{(s)}}{x-z_s}-\sum_{j=1}^m\frac{\tl w_j\cdot 0}{x-t_j}
			\eeq
			for some $m\in\Z_{\gge 0}$, $w_s\in\mathcal W$ for $s=1,\dots,n$ and $ \tl w_j\in\mathcal W$, $t_j\in\bP^1\setminus  \bm z$ for $j=1,\dots,m$,
			\item there exists $w_\infty\in\mathcal W$ such that
			\beq\label{eq weight infty}
			\sum_{s=1}^n w_s\cdot\cla^{(s)}+\sum_{j=1}^m \tl w_j\cdot 0=w_\infty\cdot \cla,
			\eeq
			\item $[\nabla]$ is regular at $t_j$ for $j=1,\dots,m$.
		\end{enumerate}\qed
	\end{lem}

\begin{rem}\label{rem sum even}
	The condition \eqref{eq weight infty} implies that $\sum_{s=1}^n\langle \alpha_r,\cla^{(s)}\rangle+\langle \alpha_r,\cla\rangle$ is even if $N=2r$.
\end{rem}
	\subsection{Miura transformation}
	Following \cite{DS}, one can associate a linear differential operator $L_{\nabla}$ to each Miura $\g$-oper $\nabla=\pa_x+p_{-1}+\bs v(x)$, $\bs v(x)\in \cM(\h)$.
	
	In the case of $\sln_{r+1}$, $\bs v(x)\in \cM(\h)$ can be viewed as an $(r+1)$-tuple $(v_1(x),\dots,v_{r+1}(x))$ such that $\sum_{i=1}^{r+1}v_i(x)=0$. The {\it Miura transformation} sends $\nabla=\pa_x+p_{-1}+\bs v(x)$ to the  operator
	\[L_{\nabla}=(\pa_x+v_1(x))\dots(\pa_x+v_{r+1}(x)).\]Similarly, the Miura transformation takes the form
	\[
	L_{\nabla}=(\pa_x+v_1(x))\dots(\pa_x+v_{r}(x))(\pa_x-v_{r}(x))\dots(\pa_x-v_{1}(x))
	\]for $\g=\mathfrak{sp}_{2r}$ and
	\[
	L_{\nabla}=(\pa_x+v_1(x))\dots(\pa_x+v_{r}(x))\pa_x(\pa_x-v_{r}(x))\dots(\pa_x-v_{1}(x))
	\]for $\g=\mathfrak{so}_{2r+1}$. The formulas of the corresponding linear differential operators for the cases of $\mathfrak{sp}_{2r}$ and $\mathfrak{so}_{2r+1}$ can be understood with the embeddings described in Sections \ref{sec id C} and \ref{sec id B}.
	
	It is easy to see that different representatives of $[\nabla]$ give the same differential operator, we can write this map as $[\nabla]\mapsto L_{[\nabla]}$.
	
	Recall the definition of $(\cla)_{ii}$ for $\cla\in\h$ from Sections \ref{sec id C} and \ref{sec id B}.
	\begin{lem}\label{lem exps}
		Suppose $\nabla$ is a Miura $\g$-oper with $[\nabla]\in\Opg(\bP^1)_{\check\bLa,\bm{z}}$, then $L_{[\nabla]}$ is a monic Fuchsian differential operator with singularities at points in $\bm z$ only. The exponents of $L_{[\nabla]}$ at $z_s$, $s=1,\dots,n$, are $(\cla^{(s)})_{ii}+N-i$, and the exponents at $\infty$ are $-(\cla)_{ii}-N+i$, $i=1,\dots,N$.
	\end{lem}
	\begin{proof}
		Note that $\nabla$ satisfies the conditions (i)-(iii) in Lemma \ref{lem m-form}. By Theorem 5.11 in \cite{F} and Lemma \ref{lem res id}, we can assume $w_s=1$ for given $s$. The lemma follows directly.
	\end{proof}
	Denote by $Z(\mathscr G)$ the center of $\mathscr G$, then
	\[
	Z(\mathscr G)=\begin{cases}\{I_{2r+1}\} &\text{if }\g=\mathfrak{so}_{2r+1},\\\{\pm I_{2r}\}&\text{if }\g=\mathfrak{sp}_{2r}.\end{cases}
	\]We have the following lemma.
	\begin{lem}\label{lem monodromy}
		Suppose $\nabla$ is a Miura $\g$-oper with $[\nabla]\in\Opg(\bP^1)_{\check\bLa,\bm{z}}$. If $\g=\mathfrak{so}_{2r+1}$, then $L_{[\nabla]}$ is a monodromy-free differential operator. If $\g=\mathfrak{sp}_{2r}$, then the monodromy of $L_{[\nabla]}$ around
		$z_s$ is $-I_{2r}$ if and only if $\langle \alpha_r,\cla^{(s)}\rangle$ is odd for given $s\in\{1,\dots,n\}$.\qed
	\end{lem}
	\subsection{Relations with pure self-dual spaces}
	Let $\check\bLa=(\cla^{(1)},\dots,\cla^{(n)},\cla)$ be a sequence of $n+1$ dominant integral $\g$-coweights and let $\bs z=(z_1,\dots,z_n,\infty)\in {\mathring{\bP}_{n+1}}$.
	
	Consider $\check\bLa$ as a sequence of dominant integral $\g_N$-weights. Choose $d$ large enough so that $k:=d-N-\sum_{s=1}^n(\cla^{(s)})_{11}-(\cla)_{11}\gge 0$. (We only need to consider the case that $\sum_{s=1}^n(\cla^{(s)})_{11}+(\cla)_{11}$ is an integer for $N=2r$, see Lemma \ref{lem sym weight new} and Remark \ref{rem sum even}.) Let $\bm k=(0,\dots,0,k)$. Note that we always have $|\check\bLa_{A,\bm k}|=N(d-N)$ and spaces of polynomials in $\s\Omega_{\check\bLa,\bm k,\bm z}$ ($=\s\Omega_{\check\bLa_{A,\bm k},\bm z}$) are pure self-dual spaces.
	
	\begin{thm}\label{thm bij oper sd}
		There exists a bijection between $\Opg(\mathbb{P}^1)_{\check\bLa,\bm{z}}$ and $\s\Omega_{\check\bLa,\bm k,\bm z}$ given by the map $[\nabla]\mapsto \Ker (f^{-1}\cdot L_{[\nabla]}\cdot f)$, where $\bs T=(T_1,\dots,T_{N})$ is associated with $\check\bLa_{A,\bm k},\bs z$ and $f=(T_1\dots T_{N})^{-1/2}$.
	\end{thm}
	\begin{proof}
		We only prove it for the case of $\g=\mathfrak{sp}_{2r}$. Suppose $[\nabla]\in \Opg(\mathbb{P}^1)_{\check\bLa,\bm{z}}$, by Lemmas \ref{lem sur} and \ref{lem m-form}, we can assume $\nabla$ has the form \eqref{eq m-form} satisfying the conditions (i), (ii), and (iii) in Lemma \ref{lem m-form}.
		
		Note that if $\langle \alpha_r,\cla^{(s)}\rangle$ is odd, $f$ has monodromy $-I_{2r}$ around the point $z_s$. By Lemma \ref{lem monodromy}, one has that $f^{-1}\cdot L_{[\nabla]}\cdot f$ is monodromy-free around the point $z_s$ for $s=1,\dots,n$. Note also that $\sum_{s=1}^n\langle \alpha_r,\cla^{(s)}\rangle+\langle \alpha_r,\cla\rangle$ is even, it follows that $f^{-1}\cdot L_{[\nabla]}\cdot f$ is also monodromy-free around the point $\infty$. Hence $f^{-1}\cdot L_{[\nabla]}\cdot f$ is a monodromy-free differential operator. 
		
		It follows from Lemmas \ref{lem exp} and \ref{lem exps} that $\Ker (f^{-1}\cdot L_{[\nabla]}\cdot f)\in\Omega_{\check\bLa_{A,\bm k},\bm z}$. Since $L_{[\nabla]}$ takes the form \[(\pa_x+v_1(x))\dots(\pa_x+v_{r}(x))(\pa_x-v_{r}(x))\dots(\pa_x-v_{1}(x)),\]
		it follows that $\Ker (f^{-1}\cdot L_{[\nabla]}\cdot f)$ is a pure self-dual space by Lemma \ref{lem dual-oper}.
		
		If there exist $[\nabla_1],[\nabla_2]\in\Opg(\mathbb{P}^1)_{\check\bLa,\bm{z}}$ such that $f^{-1}\cdot L_{[\nabla_1]}\cdot f=f^{-1}\cdot L_{[\nabla_2]}\cdot f$, then they are the same differential operator constructed from different bases of $\Ker (f^{-1}\cdot L_{[\nabla]}\cdot f)$ as described in Lemma \ref{lem D_X}. Therefore they correspond to the same $\mathfrak{so}_{2r+1}$-population by Theorem 7.5 of \cite{MV}. It follows from Theorem 4.2 and remarks in Section 4.3 of \cite{MV2} that $[\nabla_1]=[\nabla_2]$.
		
		Conversely, give a self-dual space $X\in\s\Omega_{\check\bLa,\bm k,\bm z}$. By Lemma \ref{lem y}, there exists a basis $\Gamma$ of $X$ such that for $\bs y_{\Gamma}=(y_1,\dots,y_{N-1})$ we have $y_i=y_{N-i}$, $i=1,\dots,N-1$. Following \cite{MV2}, define $\bs v\in\cM(\h)$ by
		\[
		\langle \alpha_i,\bs v\rangle=-\ln'\Big(T_i\prod_{j=1}^r y_j^{-a_{i,j}}\Big),
		\]then we introduce the Miura $\g$-oper $\nabla_\Gamma=\pa_x+p_{-1}+\bs v$, which only has  regular singularities. It is easy to see from Lemma \ref{lem D_X} that $ f^{-1}\cdot L_{[\nabla_\Gamma]}\cdot f=\mathcal D_X$. It follows from the same argument as the previous paragraph that $[\nabla_\Gamma]=[\nabla_{\Gamma'}]$ for any other basis $\Gamma'$ of $X$ and hence $[\nabla_\Gamma]$ is independent of the choice of $\Gamma$. Again by Lemma \ref{lem y}, for any $x_0\in \C\setminus\bs z$ we can choose $\Gamma$ such that $y_i(x_0)\ne 0$ for all $i=1,\dots,N-1$, it follows that $[\nabla_\Gamma]$ is regular at $x_0$. By exponents reasons, see Lemma \ref{lem exps}, we have
		\[
		\text{res}_{z_s}[\nabla_\Gamma]=[\cla^{(s)}]_{\mathcal W},\quad \text{res}_{\infty}[\nabla_\Gamma]=-[\cla]_{\mathcal W},\quad s=1,\dots,n.
		\]On the other hand, $[\nabla_\Gamma]$ is monodromy-free by Theorem 4.1 of \cite{MV2}. It follows that $[\nabla_\Gamma]\in\Opg(\mathbb{P}^1)_{\check\bLa,\bm{z}}$, which completes the proof.
	\end{proof}
	
	\section{Proof of main theorems}\label{sec proof}
	\subsection{Proof of Theorems \ref{thm simple deg A} and \ref{thm A strata}}\label{sec proof A}We prove Theorem \ref{thm simple deg A} first. 
	
	By assumption, $\bm\Xi=(\xi^{(1)},\dots,\xi^{(n-1)})$ is a simple degeneration of $\bLa=(\la^{(1)},\dots,\la^{(n)})$. Without loss of generality, we assume that $\xi^{(i)}=\la^{(i)}$ for $i=1,\dots,n-2$ and
	\[
	\dim(V_{\la^{(n-1)}}\otimes V_{\la^{(n)}})_{{\xi}^{(n-1)}}^\sing >0.
	\]
	
	Recall the strata $\Omega_{\bLa}$ is a union of intersections of Schubert cells $\Omega_{\bLa,\bm z}$, see \eqref{strata}.
	Taking the closure of $\Omega_\bLa$ is equivalent to allowing coordinates of $\bm z\in  {\mathring{\bP}}_n$ coincide. 
	
	Let $\bm z_0=(z_1,\dots,z_{n-1})\in {\mathring{\bP}}_{n-1}$. Let $X\in\Omega_{\bm{\Xi},\bm z_0}$. By Theorem \ref{thm bijection}, there exists a common eigenvector $v\in (V_{\bm{\Xi},\bm z_0})^{\sln_N}$ of the Bethe algebra $\mc B$ such that $\mc D_v=\mc D_{X}$. 
	
	Let $\bm z_0'=(z_1,\dots,z_{n-1},z_{n-1})$. Consider the $\mc B$-module $V_{\bLa,\bm z_0'}$, then we have
	\begin{align*}
	V_{\bLa,\bm z'_0}=&(\bigotimes_{s=1}^{n-2} V_{\la^{(s)}}(z_s))\otimes (V_{\la^{(n-1)}}\otimes V_{\la^{(n)}})(z_{n-1})\\
	=&\bigoplus_{\mu}c_{\la^{(n-1)},\la^{(n)}}^{{\mu}}(\bigotimes_{s=1}^{n-2} V_{ \la^{(s)}}(z_s))\otimes V_{{\mu}}(z_{n-1}),
	\end{align*}
	where $c_{\la^{(n-1)},\la^{(n)}}^{{\mu}}:=\dim(V_{\la^{(n-1)}}\otimes V_{\la^{(n)}})_{\mu}^\sing$ are the Littlewood-Richardson coefficients. Since $\dim(V_{\la^{(n-1)}}\otimes V_{\la^{(n)}})_{{\xi}^{(n-1)}}^\sing >0$, we have $V_{\bm{\Xi},\bm z_0}\subset V_{\bLa,\bm z'_0}$. In particular, $(V_{\bm{\Xi},\bm z_0})^{\sln_N}\subset (V_{\bLa,\bm z'_0})^{\sln_N}$. Hence $v$ is a common eigenvector of the Bethe algebra $\mc B$ on $(V_{\bm{\La},\bm z'_0})^{\sln_N}$ such that $\mc D_v=\mc D_X$. 
	
	It follows that $X$ is a limit point of $\Omega_{\bLa, \bs z}$ as $z_n$ approaches $z_{n-1}$. 
	This completes the proof of Theorem \ref{thm simple deg A}.
	
	Theorem \ref{thm A strata} follows directly from Theorem \ref{thm simple deg A}.
\subsection{Proof of Theorems \ref{bi rep sgr}, \ref{thm simple deg BC}, and \ref{thm BC strata}}\label{sec proof BC} We prove Theorem \ref{bi rep sgr} first. We follow the convention of Section \ref{sec oper}.
	
We can identify the sequence $\check\bLa=(\cla^{(1)},\dots,\cla^{(n)},\cla)$ of dominant integral $\g$-coweights as a sequence of dominant integral $\g_N$-weights. Consider the $\g_N$-module
	$V_{\check\bLa}=V_{\cla^{(1)}}\otimes\dots\otimes V_{\cla^{(n)}}\otimes V_{\cla}$. It follows from Theorem 3.2 and Corollary 3.3 of \cite{R} that there exists a bijection between the joint eigenvalues of the $\g_N$ Bethe algebra $\cB$ acting on $(V_{\cla^{(1)}}(z_1)\otimes\dots\otimes V_{\cla^{(n)}}(z_n))^\sing$  and the $\g$-opers in $\Opg(\mathbb{P}^1)_{\check\bLa,\bm{z}}$ for all possible dominant integral $\g$-coweight $\cla$. In fact, one can show that Theorem 3.2 and Corollary 3.3 of \cite{R} are also true for the subspaces of $(V_{\cla^{(1)}}(z_1)\otimes\dots\otimes V_{\cla^{(n)}}(z_n))_{\cla}^\sing$ with specific $\g_N$-weight $\cla$. Recall that $\bm k=(0,\dots,0,k)$, where $k=d-N-\sum_{s=1}^n(\cla^{(s)})_{11}-(\cla)_{11}\gge 0$. Since one has the canonical isomorphism of $\cB$-modules
	\[
	(V_{\check{\bLa},\bm z})^{\g_N}\cong(V_{\cla^{(1)}}(z_1)\otimes\dots\otimes V_{\cla^{(n)}}(z_n))_{\cla}^\sing,
	\]
	by Theorem \ref{thm bij oper sd}, we have the following theorem.
	\begin{thm}\label{thm bij selfdual} 
		There exists a bijection between the joint eigenvalues of the $\g_N$ Bethe algebra $\mc B$ acting on $(V_{\check\bLa,\bm z})^{\g_N}$ and $\s\Omega_{\check\bLa,\bm k,\bs z}\subset\s\Gr(N,d)$ such that given a joint eigenvalue of $\mc B$ with a corresponding $\mc B$-eigenvector $v$ in $(V_{\check\bLa,\bm z})^{\g_N}$ we have ${\rm{Ker}}\ ((T_1\dots T_{N})^{1/2}\cdot\mc D_v \cdot (T_1\dots T_{N})^{-1/2})\in \s\Omega_{\check\bLa,\bm k,\bs z}$.\qed
	\end{thm}

The fact that ${\rm{Ker}}\ ((T_1\dots T_{N})^{1/2}\cdot\mc D_v \cdot (T_1\dots T_{N})^{-1/2})\in \s\Omega_{\check\bLa,\bm k,\bs z}$
for the eigenvector $v\in (V_{\check\bLa,\bm z})^{\g_N}$ of the $\g_N$ Bethe algebra (except for the case of even $N$ when there exists $s\in\{1,2,\dots,n\}$
such that $\langle \alpha_r,\cla^{(s)}\rangle $ is odd) also follows from the results of \cite{LMV} and \cite{MM}.

Note that by Proposition 2.10  in \cite{R}, the $i$-th coefficient of the scalar differential operator $L_{[\nabla]}$ in Theorem \ref{thm bij oper sd} is obtained by action of a universal series $G_i(x)\in \mc U(\g_N[t][[x^{-1}]])$.
Theorem \ref{bi rep sgr} for the case of $N\gge 4$ is a direct corollary of Theorems \ref{thm bij oper sd} and \ref{thm bij selfdual}.

Thanks to Theorem \ref{bi rep sgr}, Theorems \ref{thm simple deg BC} and \ref{thm BC strata} can be proved in a similar way as Theorems \ref{thm simple deg A} and \ref{thm A strata}.

	\begin{appendix}
		\section{Self-dual spaces and  $\varpi$-invariant vectors}
		\subsection{Diagram automorphism $\varpi$}
		There is a diagram automorphism $\varpi:\sln_{N}\to \sln_{N}$ such that
		\[\varpi(E_i)=E_{N-i},\quad \varpi(F_i)=F_{N-i},\quad \varpi^2=1,\quad\varpi(\h_A)=\h_A.\]The automorphism $\varpi$ is extended to the automorphism of $\gl_N$ by
		\[
		\gl_N\to\gl_N,\quad e_{ij}\mapsto (-1)^{i-j-1}e_{N+1-j,N+1-i},\quad i,j=1,\dots,N.
		\]
	 By abuse of notation, we denote this automorphism of $\gl_N$ also by $\varpi$.
		
		The restriction of $\varpi$ to the Cartan subalgebra $\h_A$ induces a dual map $\varpi^*:\h_A^*\to \h_A^*$, $\la\mapsto \la^\star$, by
		\[
		\la^\star(h)=\varpi^*(\lambda)(h)=\lambda (\varpi(h)),
		\]
		for all $\lambda\in\h_A^*,h\in \h_A$.
		
		Let $(\h_A^*)^0=\{\lambda\in\h_A^*~|~\la^\star=\la\}\subset\h_A^*$. We call elements of $(\h_A^*)^0$ \emph{symmetric weights}.
		
		Let $\h_N$ be the Cartan subalgebra of $\g_N$. Consider the root system of type $\mathrm{A}_{N-1}$ with simple roots  $\alpha_1^A,\dots,\alpha_{N-1}^A$ and the root system of  $\g_N$ with simple roots $\alpha_1,\dots,\alpha_{\left[\frac{N}{2}\right]}$.
		
		There is a linear isomorphism $P_\varpi^*:\h_N^*\to (\h_A^*)^0$, $\la\mapsto\la_A$, where $\la_A$ is defined by
		\beq\label{eq A weight}
		\langle \la_A,\calpha_i^A\rangle =\langle \la_A,\calpha_{N-i}^A\rangle=\langle \la,\calpha_i\rangle,\quad i=1,\dots,\left[\frac{N}{2}\right].
		\eeq
		
		Let $\lambda\in \h_A^*$ and fix two nonzero highest weight vectors $v_\la\in (V_\la)_\la,v_{\la^\star}\in (V_{\la^\star})_{\la^\star}$. Then there exists a unique linear isomorphism $\mc I_\varpi: V_\la\to V_{\la^\star}$ such that
		\beq\label{eq twining map}
		\mc I_\varpi (v_\la)=v_{\la^\star},\quad \mc I_\varpi(gv)=\varpi(g)\mc I_\varpi(v),
		\eeq
		for all $g\in\sln_N,v\in V_\la$.
		In particular, if $\la$ is a symmetric weight, $\mc I_\varpi$ is a linear automorphism of $V_\la$, where we always assume that $v_\la=v_{\la^\star}$.
		
		Let $M$ be a finite-dimensional $\sln_N$-module with a weight space decomposition $M=\bigoplus_{\mu\in \h_A^*}(M)_\mu $. Let $f: M\to M$ be a linear map such that $f(hv)=\varpi(h)f(v)$ for $h\in\h_A,v\in M$. Then it follows that $f((M)_\mu)\subset (M)_{\mu^\star}$ for all $\mu\in \h_A^*$. Define a formal sum
		\[
		\mathrm{Tr}_M^{\varpi} f =\sum_{\mu\in (\h_A^*)^0} \mathrm{Tr}(f|_{(M)_\mu})e(\mu),
		\]
		where $\mathrm{Tr}(f|_{(M)_\mu})$ for $\mu\in (\h_A^*)^0$ denotes the trace of the restriction of $f$ to the weight space $(M)_\mu$.

		\begin{lem}\label{lem twining char product}
			We have $
			\mathrm{Tr}_{M\otimes M'}^{\varpi}(f\otimes f')=(\mathrm{Tr}_{M}^{\varpi}f)\cdot(\mathrm{Tr}_{M'}^{\varpi}f')$.\qed
		\end{lem}
		Let $\bLa=(\la^{(1)},\dots,\la^{(n)})$ be a sequence of dominant integral $\g_N$-weights, then the tuple $\bLa^A=(\la_A^{(1)},\dots,\la_A^{(n)})$ is a sequence of symmetric dominant integral $\sln_N$-weights. Let $V_{\bLa^A}=\bigotimes_{s=1}^n V_{\la_A^{(s)}}$. The tensor product of maps $\mc I_\varpi$ in \eqref{eq twining map} with respect to $\la_A^{(s)}$, $s=1,\dots,n$,  gives a linear isomorphism
		\beq\label{eq tensor-twin}
		\mc I_\varpi:V_{\bLa^A}\to V_{\bLa^A},
		\eeq
		of $\sln_{N}$-modules. Note that the map $\mc I_\varpi$ preserves the weight spaces with symmetric weights and the corresponding spaces of singular vectors. In particular,
		$(V_{\bLa^A})^{\sln_N}$ is invariant under $\mc I_\varpi$.
		\begin{lem}\label{lem dimension of singular space}
			Let $\mu$ be a $\g_N$-weight. Then we have $$\dim(V_\bLa)^{\sing}_\mu = \Tr\big({\mc I_\varpi|_{ (V_{\bLa^A})^{\sing}_{\mu_A}}}\big),\qquad \dim(V_\bLa)_\mu = \Tr\big({\mc I_\varpi|_{ (V_{\bLa^A})_{\mu_A}}}\big).$$
			In particular,
			$\dim(V_\bLa)^{\g_N} = \Tr\big({\mc I_\varpi|_{ (V_{\bLa^A})^{\sln_N}}}\big)$.
		\end{lem}
		
		\begin{proof}
			The statement follows from Lemma \ref{lem twining char product} and Theorem 1 of Section 4.4 of \cite{FSS}.
		\end{proof}
		
		\subsection{Action of $\varpi$ on the Bethe algebra}
		The automorphism $\varpi$ is extended to the automorphism of current algebra $\gl_N[t]$ by the formula $\varpi(g\otimes t^s)=\varpi(g)\otimes t^s$, where $g\in\gl_N$ and $s=0,1,2,\dots\ $. Recall the operator $\cDB$, see \eqref{eq rdet expanded}.
		\begin{prop}\label{prop B invol}
			We have the following identity
			\[
			\varpi(\mc D^{\mc B})=\pa_x^N+\sum_{i=1}^N (-1)^i\pa_x^{N-i}B_i(x).
			\]
		\end{prop}
		
		\begin{proof}
			It follows from the proof of Lemma 3.5 of \cite{BHLW} that no nonzero elements of $\mc U(\gl_N[t])$ kill all $\bigotimes_{s=1}^n L(z_s)$ for all $n\in \Z_{>0}$ and all $z_1,\dots, z_n$. It  suffices to show the identity when it evaluates on $\bigotimes_{s=1}^n L(z_s)$. 
			
			Following the convention of \cite{MTV6}, define the $N\times N$ matrix
			$\mc G_h=\mc G_h(N,n,x,p_x,\bs z,\bs \la,X,P)$ by the formula
			\[
			\mc G_h:=
			\Big(\,(p_x-\la_i)\,\delta_{ij}+
			\sum_{a=1}^n(-1)^{i-j}\frac{x_{N+1-i,a}p_{N+1-j,a}}{x-z_a}\,\Big)_{i,j=1}^N.
			\]By Theorem 2.1 of \cite{MTV6}, it suffices to show that
			\beq\label{eq rdet identity}
			\on{rdet}(\mc G_h)\prod_{a=1}^n(x-z_a)=\hspace{-10pt}
			\sum_{A,B, |A|=|B|} \prod_{b\not \in A}(p_x-\la_b)\prod_{a\not\in B}(x-z_a)
			\det(x_{ab})_{a\in A}^{b\in B} \
			\det(p_{ab})_{a\in A}^{b\in B}.
			\eeq
			The proof of \eqref{eq rdet identity} is similar to the proof of Theorem 2.1 in \cite{MTV6} with the following modifications.
			
			Let $m$ be a product whose factors are of the form $f(x)$,
			$p_x$, $p_{ij}$, $x_{ij}$ where $f(x)$ is a rational function in $x$. Then the product $m$ will be called \emph{normally ordered} if all factors
			of the form $p_x$, $x_{ij}$ are on the left from all factors of the form $f(x)$, $p_{ij}$.
			
			Correspondingly, in Lemma 2.4 of \cite{MTV6}, we put the normal order for the first $i$ factors of each summand. 
		\end{proof}

		We have the following corollary of Proposition \ref{prop B invol}.
		
		\begin{cor} The $\gl_N$ Bethe algebra $\mc B$ is invariant under $\varpi$, that is $\varpi (\mc B)= \mc B$.\qed
	\end{cor}
	Let $\bLa=(\la^{(1)},\dots,\la^{(n)})$ be a sequence of partitions with at most $N$ parts and $\bm z=(z_1,\dots,z_n)\in {\mathring{\bP}}_n$. 
	
	Let $v\in (V_{\bLa,\bm z})^{\sln_N}$ be an eigenvector of the $\gl_N$ Bethe algebra $\mc B$. Denote the $\varpi(\mc D^{\mc B})_v$ the scalar differential operator obtained by acting by the formal operator $\varpi(\mc D^{\mc B})$ on $v$.
		
		\begin{cor}\label{cor omega}
			Let $v\in (V_{\bLa,\bm z})^{\sln_N}$ be a common eigenvector of the $\gl_N$ Bethe algebra; then the identity $\varpi(\mc D^{\mc B})_v=\big(\mc D_v\big)^*$ holds.\qed
		\end{cor}
		
		 Let $\bm\Xi=(\xi^{(1)},\dots,\xi^{(n)})$ be a sequence of $N$-tuples of integers. Suppose
		\[\xi^{(s)}-\la^{(s)}=m_s(1,\dots,1),\quad s=1,\dots,n.\]
		Define the following rational functions depending on $m_s$, $s=1,\dots,n$,
		\[\varphi(x)=\prod_{s=1}^n(x-z_s)^{m_s},\quad \psi(x)=\ln'(\varphi(x))=\sum_{s=1}^n\frac{m_s}{x-z_s}.\]Here we use the convention that $1/(x-z_s)$ is considered as the constant function $0$ if $z_s=\infty$.
		\begin{lem}\label{lem shifted auto}
			For any formal power series $a(x)$ in $x^{-1}$ with complex coefficients, the linear map obtained by sending $e_{ij}(x)$ to $e_{ij}(x)+\delta_{ij}a(x)$ induces an automorphism of $\gl_N[t]$.\qed
		\end{lem}
		We denote the automorphism  in Lemma \ref{lem shifted auto} by $\eta_{a(x)}$.
	
		\begin{lem}\label{lem shifted module}
			The $\cB$-module obtained by pulling $V_{\bLa,\bm z}$ via $\eta_{\psi(x)}$ is isomorphic to $V_{\bm\Xi,\bm z}$.\qed
		\end{lem}
		
		By Lemma \ref{lem shifted module}, we can identify the $\cB$-module $V_{\bm\Xi,\bm z}$ with the $\cB$-module $V_{\bLa,\bm z}$ as vector spaces. This identification is an isomorphism of $\sln_N$-modules.
		For $v\in (V_{\bLa,\bm z})^{\sln_N}$ we use $\eta_{\psi(x)}(v)$ to express the same vector in $(V_{\bm\Xi,\bm z})^{\sln_N}$ under this identification. 
		
		\begin{lem}\label{lem conjugation operator}
			The following identity for differential operators holds
			\[
			\eta_{\psi(x)}(\cDB)=\varphi(x)\cDB (\varphi(x))^{-1}.
			\]
		\end{lem}
		\begin{proof}
			The lemma follows from the simple computation: \[\varphi(x)(\pa_x-e_{ii}(x))(\varphi(x))^{-1}=\pa_x-e_{ii}(x)-\psi(x).\]
		\end{proof}
		
		\begin{prop}\label{prop shift}
			Let $v\in (V_{\bLa,\bm z})^{\sln_N}$ be an eigenvector of the Bethe algebra such that $\mc D_v=\mc D_X$ for some $X\in\Omega_{\bLa,\bm z}$, then $\mc D_{\eta_{\psi(x)}(v)}=\mc D_{\varphi(x)\cdot  X}$.
		\end{prop}
		\begin{proof}
			With the identification between the $\mc B$-modules $V_{\bm\Xi,\bm z}$ and $V_{\bLa,\bm z}$, we have
			\[\mc D_{\eta_{\psi(x)}(v)}=\big(\eta_{\psi(x)}(\cDB)\big)_v=\varphi(x)\mc D_v (\varphi(x))^{-1}=\varphi(x)\mc D_X (\varphi(x))^{-1}=\mc D_{\varphi(x)\cdot  X}.\]
			The second equality follows from Lemma \ref{lem conjugation operator}.
		\end{proof}

		\subsection{$\mc I_\varpi$-invariant Bethe vectors and self-dual spaces}Let $\bLa=(\la^{(1)},\dots,\la^{(n)})$ be a tuple of dominant integral $\g_N$-weights. Recall the map $\mc I_\varpi:V_{\bLa^A}\to V_{\bLa^A},$ from \eqref{eq tensor-twin}. 
		
		Note that an $\sln_N$-weight can be lifted to a $\gl_N$-weight such that the $N$-th coordinate of the corresponding $\gl_N$-weight is zero. From now on, we consider $\la_A^{(s)}$ from \eqref{eq A weight} as $\gl_N$-weights obtained from \eqref{eq partition BC}, that is as the partitions with at most $N-1$ parts.
		
		Let $\bm\Xi=(\xi^{(1)},\dots,\xi^{(n)})$ be a sequence of $N$-tuples of integers such that
		\[
		\xi^{(s)}-\la_A^{(s)}=-(\la_A^{(s)})_1(1,\dots,1),\quad s=1,\dots,n.
		\]
		Consider the $\sln_{N}$-module $V_{\bLa^A}$ as the $\gl_{N}$-module $V_{\bLa_A}$, the image of $V_{\bLa_A}$ under $\mc I_\varpi$ in \eqref{eq tensor-twin}, considered as a $\gl_N$-module, is $V_{\bm\Xi}$. Furthermore,
		the image of $(V_{\bLa_A})^{\sln_N}$ under $\mc I_\varpi$ is $(V_{\bm\Xi})^{\sln_N}$.
		
		Let $\bm T=(T_1,\dots,T_N)$ be associated with $\bLa_A,\bm z$, we have
		\[
		T_1\cdots T_N=\prod_{s=1}^n(x-z_s)^{(\la_A^{(s)})_1}.
		\]Let $\varphi(x)=T_1\cdots T_N$ and let $\psi(x)=\varphi'(x)/\varphi(x)$. Hence by Lemma \ref{lem shifted module}, the pull-back of $V_{\bm\Xi,\bm z}$ through $\eta_{\psi(x)}$ is isomorphic to $V_{\bLa_A,\bm z}$. Furthermore, the pull-back of $(V_{\bm\Xi,\bm z})^{\sln_N}$ through $\eta_{\psi(x)}$ is isomorphic to $(V_{\bLa_A,\bm z})^{\sln_N}$.
		
		\begin{thm}\label{thm w-invariant}
			Let $v\in (V_{\bLa_A,\bm z})^{\sln_N}$ be an eigenvector of the $\gl_N$ Bethe algebra $\mc B$ such that $\mc D_v=\mc D_X$ for some $X\in\Omega_{\bLa_A,\bm z}$, then $\mc D_{\eta_{\psi(x)}\circ \mc I_\varpi(v)}=\mc D_{X^\dag}$. Moreover, $X$ is self-dual if and only if $\mc I_\varpi(v)=v$.
		\end{thm}
		\begin{proof}
			It follows from Proposition \ref{prop shift}, Corollary \ref{cor omega}, and Lemma \ref{lem dual-oper} that
			\begin{align*}
			\mc D_{\eta_{\psi(x)}\circ \mc I_\varpi(v)}=&\varphi(x)\mc D_{\mc I_\varpi(v)}(\varphi(x))^{-1}=\varphi(x)\varpi(\cDB)_v(\varphi(x))^{-1}\\=&(T_1\dots T_N)(\mc D_X)^*(T_1\dots T_N)^{-1}=\mc D_{X^\dag}.
			\end{align*}
			
			Since $(\la_A^{(s)})_N=0$ for all $s=1,\dots,n$, $X$ has no base points. Therefore $X$ is self-dual if and only if $\mc D_X=\mc D_{X^\dag}$. Suppose $X$ is self-dual, it follows from Theorem \ref{thm bijection} that $\eta_{\psi(x)}\circ \mc I_\varpi(v)$ is a scalar multiple of $v$. By our identification, in terms of an $\sln_N$-module homomorphism, $\eta_{\psi(x)}$ is the identity map. Moreover, since $\mc I_\varpi$ is an involution, we have $\mc I_\varpi(v)=\pm v$.
			
			Finally, generically, we have an eigenbasis of the action of $\mc B$ in $(V_{\bLa_A,\bm z})^{\sln_N}$ (for example for all $\bm z\in \R{\mathring{\bP}}_n$). In such a case, by the equality of dimensions using Lemma \ref{lem dimension of singular space}, we have $\mc I_\varpi(v)=v$. Then the general case is obtained by taking the limit.		
		\end{proof}
		
	\end{appendix}

\end{document}